\documentclass{amsart}
\usepackage[top=1.2in, bottom=1.2in, left=1.0in, right=1.0in]{geometry}

\usepackage{amsmath}
\usepackage{paralist}
\usepackage{enumitem}
\usepackage{graphicx}
\usepackage{amsthm}
\graphicspath{{./figures/}}
\usepackage{psfrag}
\let\div\undefined
\DeclareMathOperator{\div}{div}
\newcommand{\yply}{h}

\theoremstyle{plain} \newtheorem{thrm}{Theorem}[section] \newtheorem{lmm}[thrm]{Lemma} \newtheorem{crllr}[thrm]{Corollary} \newtheorem{prpstn}[thrm]{Proposition}  
\theoremstyle{definition} \newtheorem{dfntn}[thrm]{Definition}  \newtheorem{xmpl}[thrm]{Example}  \newtheorem{rmrk}[thrm]{Remark}     

\begin{document}
\title{Stability estimates for systems with small cross-diffusion}\thanks{Maria Bruna was supported by the L'Or\'eal UK and Ireland Fellowship For Women In Science.}\thanks{Luca Alasio was supported by the Engineering and Physical Sciences Research Council grant [EP/L015811/1].}

\author{Luca Alasio}
\author{Maria Bruna}
\author{Yves Capdeboscq}
\begin{abstract} 
We discuss the analysis and stability of a family of cross-diffusion boundary value problems with nonlinear diffusion and drift terms.
We assume that these systems are close, in a suitable sense, to a set of decoupled and linear problems. 
We focus on stability estimates, that is, continuous dependence of solutions with respect to the nonlinearities in the diffusion and in the drift terms.
We establish well-posedness and stability estimates in an appropriate Banach space.  Under additional assumptions we show that these estimates are time independent. 
These results apply to several problems from mathematical biology; they allow comparisons between the solutions of different models \emph{a priori}.  
For specific cell motility models from the literature, we illustrate the limit of the stability estimates we have derived numerically, and we  document  the behaviour of the solutions for extremal values of the parameters.
\end{abstract}

\subjclass{35K55, 35B30, 35Q92, 65M15}
\keywords{Cross diffusion, Continuous dependence, Quasilinear parabolic systems}
\maketitle

\section{Introduction}

\subsection{Background and motivation} \label{sec:motivation}

In this paper we analyse a class of nonlinear cross-diffusion systems of PDEs which model multi-species populations in presence of short-range interactions between individuals.
We assume that these systems are close, in a suitable sense, to decoupled sets of linear parabolic evolution problems. Such problems arise in many applications 
in mathematical biology, such as chemotactic cell migration, 
ion transport through cell membranes, and spatial segregation in interacting species. 
The strength of the interactions (and therefore of the nonlinear terms) is quantified 
with a small parameter $\epsilon$, so that when $\epsilon = 0$ the system becomes diagonal and linear. 
The biological justification 
for these models comes from weakly-interacting species, whereby interactions between populations (such as 
excluded-volume or chemotactic interactions) are present but are not dominant over the isolated species behaviour. 

The cross-diffusion systems we are interested in have the form
\begin{subequations}
	\label{eq:cross-diff_general}
\begin{equation}
	\partial_t u - \div \left [\mathfrak D(t,x,u) \nabla u - \mathfrak F(t,x,u) u \right] = 0 , \quad \text{in} \quad \Omega, t>0,
\end{equation}
with boundary and initial conditions
\begin{alignat}{2}
	\left [\mathfrak D(t,x,u) \nabla u - \mathfrak F(t,x,u) u \right] \cdot \nu &= 0, &\quad &\text{on} \quad \partial \Omega, t>0,\\
	u(0,\cdot) &= u^{0},& &\text{in} \quad \Omega,
\end{alignat}
\end{subequations}
where $\Omega$ is a smooth, bounded, and connected domain in $\mathbb R^d$ ($d= 1, 2, 3$), $\nu$ denotes 
the outward normal on $\partial \Omega$, and $u = (u_1, \dots, u_m)$ is the vector of densities of each species.
The divergence $\div$ and gradient $\nabla$ represent derivatives with respect to the $d$ spatial variables. 
Here $\mathfrak D(t, x, u)$ and $\mathfrak F(t,x,u)$ are $m\times m$ matrices of diffusion tensors and drift vectors, 
respectively (see \eqref{eq:1011} for further details).  In particular, the entries of the diffusion 
tensor  $\mathfrak D$ may be scalars in the case of isotropic diffusion, or $d\times d$ tensors in the case of anisotropic diffusion. 
The drift matrix elements $\mathfrak F_{ij}$ are $d-$dimensional vectors. In our class of cross-diffusion systems, 
the matrices $\mathfrak D$ and $\mathfrak F$ are close to  matrices that are diagonal and independent of $u$, that is, they can be written in the form
\begin{equation} \label{eq:close_diag}
\begin{aligned}
		\mathfrak D(t,x, u ) &= \mathfrak D^{(0)}(t,x) + \epsilon \mathfrak D^{(1)}(t,x,u) + O(\epsilon^2), \\
		\mathfrak F(t,x, u ) &= \mathfrak F^{(0)}(t,x) + \epsilon \mathfrak F^{(1)}(t,x,u) + O(\epsilon^2),
\end{aligned}
\end{equation}
where $\epsilon$ is a small parameter.

The focus of this paper is to study the stability of the solutions to \eqref{eq:cross-diff_general} under perturbations of order $\epsilon$. 
We establish that the solutions depend continuously on the nonlinearities $\mathfrak D^{(1)}$ and $\mathfrak F^{(1)}$ for $\epsilon$ small enough.
The cross-diffusion model \eqref{eq:cross-diff_general} is a non-linear system, and this combines two types of difficulties, namely the non-linearity
and the fact that fully coupled parabolic systems of equations do not enjoy, in general, the same smoothness
properties as parabolic equations (see, for example, \cite[Chap. 9]{giaquinta_martinazzi} , and  \cite{giaquinta1982partial}).
Our results are detailed in Proposition \ref{prop:315} and Theorem \ref{thm:297}. 
They are quantitative, in the sense that we provide a bound on $\epsilon$ below which our perturbation result applies.
The novelty of our analysis consists in the unified approach to the study of regularity and stability properties in ``strong'' 
Sobolev norms for a relatively wide class of nonlinear cross-diffusion systems.

Our stability estimate uses the underlying regularity of the system,
which, as we will see, it inherits from the leading order model, consisting of decoupled linear evolution equations. We show
that for small perturbations at least some of the regularity is preserved and, using a fixed point argument,
we deduce a stability estimate with respect to the nonlinearities of the model.

Similar results concerning nonlinear systems where interactions between species (or components) are limited to lower order term (so-called weakly coupled systems) are 
available in the work of Camilli and Marchi \cite{camilli2012continuous}. They extend the results available for scalar equations in terms of continuous 
dependence estimates in the $\sup$ norm using the doubling variable method \cite{kruzkov1970} and viscosity solutions. Their results do not apply
to fully coupled systems with cross diffusion present such as the ones we are considering. 
Continuous dependence for fully coupled quasilinear systems was studied by  Cannon, Ford and Lair \cite{cannon1976quasilinear}. They established existence and uniqueness,
following arguments of Ladyzhenskaya, Solonnikov, and Ural'tseva \cite{ladyzhenskaia1988linear} in larger Sobolev spaces (weaker norms).
They derive stability estimates under additional integrability properties assumptions for the gradients. 
We establish existence and uniqueness in stronger norms, removing the need of additional regularity assumptions. 

There are several models, especially in mathematical biology, that fit into the class of systems \eqref{eq:cross-diff_general} and \eqref{eq:close_diag}. 
This is the case for models describing the transport of cells or ions while accounting for the finite-size of particles \cite{Bruna:2012wu,Burger:2010gb,Simpson:2009gi,PerthameBook2015}. 
These models were derived from stochastic agent-based models assuming that the concentration of cells or ions is not too large, so that the transport dominates over the 
finite-size interactions between cells or ions. The diffusion and drift matrices become density-dependent due to the interactions, but this correction is small since it 
scales with the excluded volume in the system. Below we present three of such models, and show how they fit into our framework. 

\begin{xmpl}[Random walk on a lattice with size exclusion] \label{exa:model2} A cross-diffusion model for two interacting species was employed to describe the motility of 
biological cells by Simpson et al. \cite{Simpson:2009gi} or ion transport by Burger et al. \cite{Burger:2010gb}. The  models were derived assuming that particles are 
restricted to a regular square lattice and undergo a simple exclusion random walk, in which a particle can only jump to a site if it is presently unoccupied. In order to 
obtain a continuum model such as \eqref{eq:cross-diff_general} from these so-called lattice-based	 models, it is generally assumed that the occupancies of adjacent 
sides are independent, so that the jumping probabilities take a simple form and do not require correlation functions \cite{Burger:2010gb,Simpson:2009gi}. Clearly, such an 
approximation is poor when the overall occupancy of the lattice is high.
As a result, these models are generally considered valid for low-lattice occupancies. 

The models in \cite{Burger:2010gb,Simpson:2009gi} consider two species of equal size, whose diameter is given by the lattice spacing $\varepsilon$, that undergo a random 
walk with isotropic diffusion $D_i$ and external potential $V_i(x)$, for $i = 1, 2$ (the jumping rates increase with $D_i$ and the jumps are biased in the direction 
of $-\nabla V_i(x)$). There are $N_1$ particles of the first species, and $N_2$ of the second species. Under these assumptions, a cross-diffusion model of the form 
\eqref{eq:cross-diff_general} is obtained, where the population densities $u_1(t,x)$ and $u_2(t,x)$ represent the probability that a particle from first or second species 
respectively is at $x\in \Omega$ at time $t$.  The diffusion and drift matrices are given by \cite{Burger:2010gb}
\begin{subequations} \label{eq:model2} 
\begin{equation}
\mathfrak D(u)=\begin{pmatrix}D_1(1- \epsilon \bar N_2 u_2) & \epsilon D_1  \bar N_2 u_1 \\
\epsilon D_2 \bar N_1 u_2 & D_2(1-\epsilon  \bar N_1 u_1)
\end{pmatrix}, \label{eq:diff_matrix2}
\end{equation}
\begin{equation}\label{ex:model2_drift}
\mathfrak F(u)=\begin{pmatrix} -\nabla V_1(1-\epsilon  \bar N_1 u_1) & \epsilon  \bar N_2 u_1 \nabla V_1 \\
\epsilon  \bar N_1 u_2 \nabla V_2 & -\nabla V_2(1-\epsilon  \bar N_2 u_2)
\end{pmatrix},
\end{equation}
\end{subequations}
where $\epsilon = (N_1 + N_2) \varepsilon^d/|\Omega| \ll 1$ represents the total volume fraction of the lattice occupied by particles
 and $ \bar N_i = N_i/(N_1 + N_2)$.
We have written \eqref{eq:model2}  in a form consistent with our notations, which differ slightly from those used in \cite{Burger:2010gb}.
Global existence for such model was shown in \cite{Chen:2006kd}.
In that paper, as in most works using lattice-based models, the continuum model is written in terms of the volume concentrations $\hat u_i$,
so that the mass of $\hat u_i$ equals the total volume occupied by species $i$ (that is, $\int_\Omega \hat u_i(t,x) {\mathrm{d}} x = N_i \varepsilon^d/|\Omega|$).
We write \eqref{eq:model2} in terms of probability densities $u_i$, which implies that $\int_\Omega u_i {\mathrm{d}} x = 1$.
The two quantities are related by the identity $\hat u_i =  \bar N_i \epsilon u_i$.
The potentials appearing in \eqref{ex:model2_drift}, $V_i$, are not rescaled  by the diffusion coefficient as it is done in \cite{Burger:2010gb}.
The number of species can take any values provided that $\epsilon$, is small. 
The matrices in \eqref{eq:model2} are of the form \eqref{eq:close_diag} that we consider in this paper.
There are also other lattice-based models that fit well into such framework, such as that derived by Shigesada et al.
\cite{Shigesada:1979tq} to describe spatial segregation of interacting animal populations. 
\end{xmpl}

\begin{xmpl}[Brownian motion with size exclusion] \label{exa:model1} A cross-diffusion model for two interacting species of diffusive particles was obtained 
by Bruna and Chapman for $d=2,3$ in \cite{Bruna:2012wu}, starting from a system with two types of Brownian hard spheres. The population densities $u_i(t,x)$, $i=1,2$, 
represent the probability that a particle of species $i$ is at $x\in \Omega$ at time $t$, and so $\int_\Omega u_i(t,x) {\mathrm{d}} x = 1$. The model assumes 
there are $N_i$ particles of species $i$, of diameter $\varepsilon_i$ and isotropic diffusion constant $D_i$. The position $X_i$ of each particle in species $i$ 
evolves in time according to the stochastic differential equation
\begin{equation}
	\mathrm{d} X_i(t) = \sqrt{2D_i} \mathrm{d}{W}(t) -\nabla V_i(X_i(t)) \mathrm{d}t, \label{eq:sde} 
\end{equation}
where $i = 1$ or 2, and $W$ are independent, $d$-dimensional standard Brownian motions. Reflective boundary conditions are imposed whenever two particles are in 
contact ($\| X_i - X_j \| = (\varepsilon_i + \varepsilon_j)/2$, when $X_i$ and $X_j$ are of type $i$ and $j$, respectively), as well as on the boundary of the domain $\partial\Omega$.

The cross-diffusion model is derived using the method of matched asymptotic expansions under the assumption that the volume fraction of the system is small, or 
equivalently, that $(N_1 \varepsilon_1^d + N_2 \varepsilon_2^d)/|\Omega| \sim \epsilon \ll 1$, where $\epsilon$ is defined as in Example \ref{exa:model2} 
with $\varepsilon = (\varepsilon_1+\varepsilon_2)/2$. When the number of particles in each species is large, the cross-diffusion model in \cite{Bruna:2012wu} is 
of the form \eqref{eq:cross-diff_general}, with diffusion matrix
\begin{subequations} \label{model1}
\begin{equation}
\mathfrak D(u)=\begin{pmatrix}D_1(1+\epsilon a_1 u_1-\epsilon  c_1 u_2) & \epsilon D_1  b_1 u_1 \\
\epsilon D_2 b_2 u_2 & D_2(1+\epsilon  a_2 u_2 -\epsilon  c_2 u_1)
\end{pmatrix}, \label{eq:diff_matrix}
\end{equation}
and drift matrix 
\begin{equation}
\mathfrak F(u)=\begin{pmatrix} -\nabla V_1 & \epsilon c_1 \nabla( V_1- V_2)u_1 \\
\epsilon c_2 \nabla ( V_2- V_1 )u_2 & -\nabla V_2
\end{pmatrix}.\label{eq:drift_matrix}
\end{equation}
\end{subequations}
The parameters $a_i, b_i, c_i$ ($i=1,2$) are all positive numbers that depend on the problem dimension, particle sizes, numbers, and relative diffusion 
coefficients (see specific values in Section \ref{sec:numerics}). Model \eqref{model1} also fits into the form \eqref{eq:close_diag}, with $\epsilon= 0$ when 
particles are non-interactive (point particles) and evolve according to two decoupled linear drift-diffusion equations. 
\end{xmpl}

\begin{xmpl}[Asymptotic gradient-flow structures]\label{ex:model3}
Certain cross-diffusion systems possess a formal gradient-flow structure, that is, they can be formulated as
\begin{equation}
	\label{grad_flow_general}
	\partial_t u - \nabla\cdot \left ( M \nabla \frac{\delta E}{\delta u} \right) = 0, 
\end{equation}
where $M \in \mathbb R^{m\times m}$ is known as mobility matrix and $\delta E/\delta u$ is the variational derivative of the entropy (or free energy) function $E[u]$. 
While the underlying microscopic model \eqref{eq:sde} of Example \ref{exa:model1} has a natural entropy, in \cite{Bruna:2016cm} it was noted that model \eqref{model1} 
does not have an obvious gradient-flow structure, but that it is close to one that does have such convenient structure. More specifically, consider the following entropy
\begin{subequations}
	\label{eq:grad_flow1}
\begin{equation}
	E_{\epsilon}[u] = \int_{\Omega} \bigg [ u_1\log u_1 + u_2\log u_2 + u_1 \frac{V_1}{D_1} + u_2 \frac{V_2}{D_2} 
	+\frac{\epsilon}{2} \left(  a_1 u_1^2 +2 a_{12} u_1 u_2 + a_2 u_2^2\right)  \bigg ] {\mathrm{d}} x,\label{eq:entropy_general}
\end{equation}
with $ a_{12} = (d-1) ( c_1 +  c_2)$, 
and the mobility matrix 
\begin{equation}
M_\epsilon(u)=\begin{pmatrix}D_1 u_1(1-\epsilon c_1 u_2) & D_1 c_2\epsilon u_1 u_2\\
D_2 c_1 \epsilon u_1 u_2 & D_2 u_2(1- c_2\epsilon u_1)
\end{pmatrix}.\label{eq:mobility_general}
\end{equation}
\end{subequations}
The cross-diffusion system \eqref{eq:cross-diff_general} with diffusion and drift matrices \eqref{model1} and $N_1 = N_2$\footnote{In \cite{Bruna:2016cm} the more general 
case when $N_1 \ne N_2$ was also considered, by writing the system in terms of number densities $N_i u_i$.}, can be rewritten as 
\begin{equation}
\partial_{t}u = \nabla\cdot\left (M_\epsilon \nabla \frac{\delta E_\epsilon }{\delta u} -\epsilon^{2}G\right),\label{gradflow_generalasy}
\end{equation}
where $G = G(u, \nabla u)$ (see more details in Section \ref{sec:numerics}). In particular, the discrepancy between the system in Example \ref{exa:model1} and 
the gradient-flow induced by \eqref{eq:grad_flow1} is of order $\epsilon^{2}$, an order higher than that of the model.\footnote{Systems \eqref{eq:cross-diff_general}-\eqref{model1} 
and \eqref{grad_flow_general}-\eqref{eq:grad_flow1} are in fact identical when both species have the same particle sizes, $\varepsilon_1 = \varepsilon_2$, 
and diffusivities, $D_1 = D_2$, since $G$ vanishes in that particular case.} Does this legitimise the use of \eqref{eq:grad_flow1} as a gradient-flow structure
of the system? Having a formal gradient-flow structure can facilitate the analysis of cross-diffusion models \cite{jungel2015boundedness}. 
The gradient-flow model \eqref{grad_flow_general}-\eqref{eq:grad_flow1} was studied in \cite{Bruna:2016cm}; stability, uniqueness of the stationary solutions, 
and a global-in-time existence result was shown.
\end{xmpl}

It is natural to ask whether the approximation argument in Example \ref{ex:model3} can be made rigorous, and, more generally if minor changes in the models can be safely ignored. 
For instance, given a two-species biological system, does it matter if we choose a lattice-based model (like in Example \ref{exa:model2}), or an off-lattice model
(like in Example \ref{exa:model1} with equal particle number, size, diffusivity, etc.)?  
If so, can we quantify the differences? Lattice-based approaches have become very common, as they offer a simple way to derive 
continuum PDE models. They can be unrealistic since most biological transport processes modelled by these  are not constrained 
on a lattice \cite{Plank:2012fa}. Nevertheless, if one is solely interested in the population-level behaviour of the system, is it worth using a more realistic off-lattice model?
When is the even simpler model (linear advection-diffusion) sufficiently accurate? 
The aim of this paper is to answer these questions and quantify the differences between models of the form \eqref{eq:cross-diff_general}.

\subsection{Outline of the results }

As we are working with systems of equations, we use different indices to refer to the ambient space variables and the component or species number.
Greek indices $1\leq\alpha,\beta\leq d$ refer to directions in the ambient space, $\mathbb{R}^{d}$, for $d=1,2,3$. Latin indices $1\leq i,j\leq m$
are used to refer to the species number.
The domain $\Omega$ where the problem is formulated is bounded, connected and of class  $C^2$ in $\mathbb{R}^{d}$. The outward normal on $\partial\Omega$ is written $\nu$.

The parabolic models we consider are weak formulations of problems of the form
\begin{align}\label{eq:1091}
\begin{aligned}
	\partial_{t}u_{i}-\partial_{\alpha}\left[\mathfrak D_{ij}^{\alpha\beta}(t,x,u)\partial_{\beta}u_{j} - \mathfrak F_{ij}^{\alpha}(t,x,u)u_{j}\right] & =  0 &\text{in }& \Omega,\\
\left[ \mathfrak D_{ij}^{\alpha\beta}(t,x,u)\partial_{\beta}u_{j} - \mathfrak F_{ij}^{\alpha}(t,x,u)u_{j}\right]\cdot\nu_{\alpha} & =  0 &\text{on }& \partial\Omega,\\
u(0,\cdot) & = u^{0} &\text{in }& \Omega,
\end{aligned}
\end{align}
for $1\leq i \leq m,$.  The Einstein summation convention is used, that is, repeated indices are implicitly summed. 

Our main result is a stability estimate for cross-diffusion systems
that are close to diagonal, decoupled, linear diffusion problems. Our reference problem
will be the weak formulation of  
\begin{align}\label{eq:ModelZero}
\begin{aligned}
	\partial_{t}u_{i}-\partial_{\alpha}\left[ D_{i}^{\alpha\beta}(t,x)\partial_{\beta}u_{i} - F_{i}^{\alpha}(t,x)u_{i}\right] &=  0 &\text{in }&\Omega,\\
\left[D_{i}^{\alpha\beta}(t,x)\partial_\beta  u_i - F_{i}^{\alpha}(t,x)u_{i}\right]\cdot \nu_{\alpha} &=  0 &\text{on }&\partial\Omega, \\
u(0,\cdot) &=  u^{0} &\text{in }& \Omega,
\end{aligned}
\end{align}
The initial datum  $u^{0}$  in \eqref{eq:1091} and \eqref{eq:ModelZero} belongs to $H^{2}(\Omega)$. Note that throughout the paper we 
write $H^2(\Omega)$ for $H^2(\Omega; \mathbb R^m)$, and similarly for other spaces.

Compared to the general system \eqref{eq:1091}, in \eqref{eq:ModelZero} we
have specified that $\mathfrak D_{ij}= \mathfrak F_{ij}=0$ if  $ i\neq j$,
and $\mathfrak D$ and $\mathfrak F$ do not depend on $u$.
In Examples \ref{exa:model2}, \ref{exa:model1}, and \ref{ex:model3}, the reference problem
corresponds to the case $\epsilon=0$, with $D_{i}^{\alpha \beta}(x,t)=\delta_{ \alpha  \beta}D_i$ and $F_i= -\nabla V_i$.
We allow time and space variations of the diffusion coefficients as it does not affect the analysis.
We could also have safely included lower-order terms, but it would have resulted
in somewhat longer and  relatively routine developments. Additionally such terms do not appear in the three examples of interest.

System \eqref{eq:ModelZero} is strongly parabolic, that is, there exist a positive constant $\lambda$
such that for every $t\in[0,\infty)$, $x\in\Omega$ and $\xi\in\mathbb{R}^{d}$,
there holds
\begin{equation}\label{eq:D-Elliptic}
D_{i}^{\alpha \beta}(t,x)\xi^{\alpha}\xi^{\beta}\geq\lambda\left|\xi\right|^{2}, \quad i=1,\ldots,m.
\end{equation}
Furthermore, we shall assume that $D$ is symmetric in the space indices $\alpha$ and $\beta$.

We allow perturbations of system \eqref{eq:ModelZero} scaled by a small parameter $\epsilon$. Namely we consider
\eqref{eq:1091} with 
\begin{align}  \label{eq:1011}
\begin{aligned}
\mathfrak D_{ij}^{\alpha\beta}(t,x,u) &=  D_{i}^{\alpha \beta}(t,x)+\epsilon a_{ij}^{\alpha\beta}(t,x)\phi_{ij}^{\alpha\beta}(u),\\
\mathfrak F_{ij}^{\alpha}(t,x,u) &=  F_{i}^{\alpha}(t,x) +\epsilon b_{ij}^{\alpha}(t,x)\psi_{ij}^{\alpha}(u).
\end{aligned}
\end{align}
The variations
of the coefficients $a$ and $b$ are of class $C^2$ in time and space, that is, 
\begin{equation}
\left\Vert \left(a,b\right)\right\Vert _{C^{2}\left([0,\infty)\times\mathbb{R}^{d}\right)}\leq M,\label{eq:LPA2}
\end{equation}
and the dependence on $u$ of the perturbations is also of class $C^2$,
\begin{equation}
\phi,\psi\in C^{2}\left(\mathbb{R}^{m}\right)^{m\times m},  \qquad \phi(0)=\psi(0)=0.\label{eq:LipschitzPerturbationAssumption}
\end{equation}
Furthermore, we assume that $D$
and $F$ satisfy the bound 
\begin{equation}
  \sum_{\alpha,\beta, i}\Vert D_{i}^{\alpha \beta}\Vert _{C^{1}\left([0,\infty)\times\mathbb{R}^{d}\right)}+\sum_{\alpha,i}\Vert F_{i}^{\alpha}\Vert _{C^{1}\left([0,\infty)\times\mathbb{R}^{d}\right)}
      \leq M.\label{eq:LinearBound}
\end{equation}
In the context of biological models, one is often interested in arbitrarily
long behaviour and, in turn, convergence to a steady state. 
Along this line, we prove sharper estimates when the coefficients $D$ and $F$ of the reference problem \eqref{eq:ModelZero} do not depend
on time and $F$ is derived from a potential
(as in Examples \ref{exa:model2}, \ref{exa:model1}, and \ref{ex:model3}). 
In particular, consider the following additional assumption:
\begin{itemize}
\item[\textbf{(H)}]  for each $i\in\{1,\ldots,m\}$, $D_i$ is independent of time and  there exists $V_i $ such that $F_{i} = - D_i \nabla V_i$. 
\end{itemize}

Our estimates will be expressed in terms of the constants appearing in assumptions \eqref{eq:D-Elliptic},
\eqref{eq:LPA2}, \eqref{eq:LipschitzPerturbationAssumption} and \eqref{eq:LinearBound}. More specifically, the following positive-valued functions will appear:
\begin{eqnarray}
L_{i}&:&R\to\left\Vert \left(\phi,\psi\right)\right\Vert _{C^{i}\left(\overline{ B_{R}(0)}\right)} \quad i=0,1,2\label{eq:DefLi},\\
K_0&:& R\to  M\left(5L_{0}\left({C_{S}^{\infty}} \, R\right)+2C_S^2 L_{1}\left({C_{S}^{\infty}}\,R\right)R\right), \label{eq:defKo} \\
K_1&:&R\to  C_S M \left(L_1(R) R +L_2(R) R^2\right),	\label{eq:defK1}\\
K_2 &:& R \to  6 R C_{T / \infty} C_S \max \left(\left(L_0(R) +  L_1(R)R \right), M ( 1+R) \right) , \label{eq:DefK2}
\end{eqnarray}
where ${C_{S}^2}$, $C_S$, and ${C_{S}^{\infty}}$ depend on $\Omega$ and $d$ and are given by \eqref{eq:defCS2}, \eqref{eq:defCSK1}, 
and \eqref{eq:defCSinfty} respectively. The constant $C_{T/\infty}$ determines the dependence on a final time $T>0$ of our estimates and is given by
\begin{equation}\label{eq:ctinfty}
  C_{T/\infty}=
 \left\{ \begin{array}{rl}
  C_T  & \textrm{ when \textbf{(H)} does not apply} \\
  C_\infty & \textrm{ when \textbf{(H)} applies}, 
  \end{array}\right.
\end{equation}
where $C_T$ is given by \eqref{eq:defC6} and depends on $M$, $\Omega$, $L_0$, $L_1$ and $T$ only, and 
$C_\infty$ is specified in \eqref{eq:C6prime} and  it depends on $M$, $\Omega$, $L_0$ and $L_1$ only -- not $T$. 
The upper bound $\epsilon_0$ on the range of values $\epsilon$ allowed will be determined by means of the following function  
\begin{equation}\label{eq:defepsilonzero}
\epsilon_0 : R\to \min \left(\frac{1}{2+2K_0(R)},\frac{1}{ 1+ K_1 (R)}\right).
\end{equation}

Our first result, which is instrumental to our main theorem, provides an existence result and a regularity
estimate for solutions of system \eqref{eq:1091}. 
Given $T>0$, we denote the parabolic cylinder by $Q_{T}=(0,T)\times\Omega$. 

\begin{dfntn}
We name  $W(Q_{T})$ the Banach space of functions with two weak derivatives in space in $L^{2}(\Omega)$ continuously in time, 
and one time derivative in $H^{1}(Q_{T})$, that is, 
$$
W\left(Q_{T}\right)  =\left\{ u\in C\left(\left[0,T\right];H^{2}(\Omega)\right), \partial_{t}u\in H^{1}
\left( Q_T\right) \right\}.
$$
\end{dfntn}

We are now ready to state our first result, concerning existence and uniqueness of solutions of \eqref{eq:1091}.

\begin{prpstn}
\label{prop:315} Assume that hypothesis
\eqref{eq:D-Elliptic},
\eqref{eq:1011},
\eqref{eq:LPA2},
\eqref{eq:LipschitzPerturbationAssumption}
and \eqref{eq:LinearBound}
hold. 
Consider $u^0 \in W(Q_T)$ satisfying the compatibility condition
\begin{equation}\label{eq:compatcondu_0}
\left[\mathfrak D_{ij}^{\alpha\beta}(t,x,u^0)\partial_{\beta}u^0_{j} - \mathfrak F_{ij}^{\alpha}(t,x,u^0)u^0_{j} \right] \cdot\nu=0 \quad \text{on }\partial\Omega,\quad i=1,\ldots,m.  
\end{equation}
Let
\begin{equation}\label{eq:DefY0}
Y_0 =   C_{T/\infty} \| u^{0}\| _{H^{2}(\Omega)}.
\end{equation}
If $\epsilon < \epsilon_0(Y_0)$,  then system \eqref{eq:1091} admits a unique solution $u\in W(Q_{T})$ and
there holds 
\[
\|u\|_{W(Q_{T})}\leq Y_0.
\]
\end{prpstn}

\begin{rmrk}
Any compatible initial data in $H^{2}(\Omega)$ is allowed, provided
$\epsilon$ is small enough.  Note that the compatibility condition \eqref{eq:compatcondu_0} holds for any initial data 
compactly supported  in $\Omega$.
All $\epsilon$ within the range $[0, \epsilon_0(Y_0))$ are allowed, and the solution $u$ is bounded linearly by its initial condition.  
When assumption \textbf{(H)} holds, the solution is bounded for all times. 

Our result holds for space dimension $d=1, 2$ and $3$, but not above. 
Two embeddings are used in our proofs: $L^4(\Omega) \subset H^1(\Omega)$, which does not hold when $d\geq5$, and  $L^\infty(\Omega)\subset H^2(\Omega)$, which does not hold when $d\geq4$. 
\end{rmrk}

Our purpose is to establish a stability result under perturbations. Therefore
we consider a second problem with $\mathfrak D$ and
$\mathfrak F$ replaced by
\begin{align} \label{eq:1011-1}
\begin{aligned}
	\widetilde {\mathfrak D}_{ij}^{\alpha\beta}(t,x,u) &=  D_i^{\alpha \beta}(t,x) +\epsilon\tilde{a}_{ij}^{\alpha\beta}(t,x)\tilde{\phi}_{ij}^{\alpha\beta}(u),\\
\widetilde {\mathfrak F}_{ij}^\alpha (t,x,u) &=  F_{i}^{\alpha}(t,x) + \epsilon\tilde{b}_{ij}^{\alpha}(t,x)\tilde{\psi}_{ij}^{\alpha}(u),
\end{aligned}
\end{align}
where $\tilde{a},$ $\tilde{b}$, $\tilde{\phi}$ and $\tilde{\psi}$
satisfy hypothesis \eqref{eq:LPA2}, \eqref{eq:LipschitzPerturbationAssumption} and, without loss of generality, 
\[
\big\Vert (\tilde\phi,\tilde\psi)\big\Vert _{C^{i}(\overline{B_R(0)})}\leq L_i(R) \quad \text{for all} \quad 0\leq R,\quad i=0,1,2.
\]
for $L_i$ defined in \eqref{eq:DefLi}. 
Our main result is as follows.
\begin{thrm}
\label{thm:297} Given $u^{0}, \tilde u^{0} \in H^{2}(\Omega)$ compactly supported in $\Omega$, 
write 
$$
Y_{1}= C_{T/\infty} \max\left(\|u^{0}\|_{H^{2}(\Omega)},\|{\tilde u}^{0}\|_{H^{2}(\Omega)}\right),
$$ 
and assume $\epsilon<\epsilon_{0} (Y_1)$ so that Proposition 
\ref{prop:315} applies for both sets of parameters. Let $u\in W(Q_{T})$
be the solution of \eqref{eq:1091} and $\tilde{u}\in W(Q_{T})$
be the solution of \eqref{eq:1091} with $\mathfrak D$,
$\mathfrak F$ 
and $u^{0}$
are replaced by $\widetilde {\mathfrak D}$ and $\widetilde {\mathfrak F}$ and ${\tilde u}^{0}$, respectively. 
Then the following stability estimate holds:
\begin{equation}
	\left\Vert \tilde{u}-u \right\Vert_{W(Q_{T})} 
    \leq 
    \Gamma_1 \| {\tilde u}^{0}-u^{0}\|_{H^{2}(\Omega)}
 +\epsilon \Gamma_2 \left(\Vert (\tilde{a},\tilde{b})-(a,b)\Vert _{C^{1}([0, \infty)\times\mathbb{R}^{d})}+\Vert (\tilde{\phi},\tilde{\psi})-(\phi,\psi)\Vert_{C^{1}\left(\overline{B_{Y_{1}}(0)}\right)}\right),
 \end{equation}
where
$\Gamma_1 = (1+ K_1(Y_1))C_{T / \infty}$,
$ \Gamma_2 = (1+K_1(Y_1))K_2(Y_1) $ and
$K_{1}$, $K_2$ are non decreasing functions given by \eqref{eq:defK1}, \eqref{eq:DefK2} respectively. They depend on $\Omega$, $M$, $\lambda$, $L_0$, $L_1$ and $L_2$ and $C_{T/\infty}$ only.
\end{thrm}

Theorem \ref{thm:297} implies, for example, that we can control the differences between the solutions of the models in Examples \ref{exa:model2} 
and \ref{exa:model1}, by considering the differences in their respective diffusion and drift matrices, which appear at order $\epsilon$. Similarly, we can also use this 
result to predict the error we will make by approximating model \eqref{model1} in Example \ref{exa:model1} as the gradient flow in Example \ref{ex:model3}. Since the differences
between models appear at order $\epsilon^2$ in this case, provided the initial data are equal, the error will be bounded and of order $\epsilon^2$ for all times (see Section \ref{sec:numerics}). 

\begin{rmrk}
 In Proposition~\ref{prop:315} the compatibility condition \eqref{eq:compatcondu_0} appears, 
 which is automatically satisfied by compactly supported initial data as we have assumed in Theorem~\ref{thm:297}. 
 However, Theorem~\ref{thm:297} also holds (with the same proof) provided that
 $u^0$ and $\tilde{u}^0$ satisfy the following four conditions:
\begin{alignat*}{3}
\left[\mathfrak D_{ij}^{\alpha\beta}(t,x,u^0)\partial_{\beta}u^0_{j} - \mathfrak F_{ij}^{\alpha}(t,x,u^0)u^0_{j} \right] \cdot\nu &=0 &\quad &\text{on }\partial\Omega,&\quad i&=1,\ldots,m,  \\
\left[\mathfrak D_{ij}^{\alpha\beta}(t,x,\tilde u^0)\partial_{\beta}\tilde{u}^0_{j} - \mathfrak F_{ij}^{\alpha}(t,x,\tilde u^0)\tilde{u}^0_{j} \right] \cdot\nu &=0 &&\text{on }\partial\Omega, & i & =1,\ldots,m,  \\
\left[\widetilde {\mathfrak D}_{ij}^{\alpha\beta}(t,x,u^0)\partial_{\beta}u^0_{j} -  \tilde{\mathfrak F}_{ij}^{\alpha}(t,x,u^0)u^0_{j} \right] \cdot\nu & =0& & \text{on }\partial\Omega,& i&=1,\ldots,m,  \\
\left[\widetilde {\mathfrak D}_{ij}^{\alpha\beta}(t,x,\tilde{u}^0)\partial_{\beta}\tilde{u}^0_{j} - \tilde{\mathfrak F}_{ij}^{\alpha}(t,x,\tilde{u}^0)\tilde{u}^0_{j} \right] \cdot\nu & =0 && \text{on }\partial\Omega, & i &=1,\ldots,m.  
 \end{alignat*}
We choose to write the result for compactly supported initial data to improve readability.
\end{rmrk}

\section{Proof of Proposition \ref{prop:315} and Theorem \ref{thm:297}}


In Lemma \ref{lem:313}, we derive an estimate for a linearisation of system \eqref{eq:1091}. 
\begin{lmm}
\label{lem:313}  
Assume that $\mathfrak D$ and $\mathfrak F$ are given by
\eqref{eq:1011}, and that $a,b$ and $\phi, \psi$ satisfy \eqref{eq:LPA2} and \eqref{eq:LipschitzPerturbationAssumption} respectively. Suppose that  $\yply \in W(Q_{T})$ satisfies
\begin{equation}\label{eq:epsilonbound}
\epsilon K_0 \left(\left\Vert \yply\right\Vert _{W\left(Q_{T}\right)}\right)<1,
\end{equation}
where $K_0$ is given by \eqref{eq:defKo}. 

For all $u^{0}\in H^{2}(\Omega)$ and $f\in C([0,T];H^1\left(Q_{T}\right))\cap H^1(0,T;L^2(\Omega)) $ such that
\begin{equation}
\left[\mathfrak D_{ij}^{\alpha\beta}(t,x,\yply)\partial_{\beta}u^0_{j} - \mathfrak F_{ij}^{\alpha}(t,x,\yply)u^0_{j}+f^\alpha_i(t=0)\right] \cdot\nu=0 \quad \text{on }\partial\Omega,\quad i=1,\ldots,m\label{eq:compatcond1}
\end{equation}
there exists a unique weak solution  $u\in W(Q_{T})$
to the linearised system
\begin{align}\label{eq:parab-linear}
\begin{aligned}
	\partial_{t}u_{i}-\partial_{\alpha}\left[\mathfrak D_{ij}^{\alpha\beta}(t,x,\yply)\partial_{\beta}u_{j} - \mathfrak F_{ij}^{\alpha}(t,x,\yply)u_{j}+f_i^\alpha\right] & =  0 & &\text{in } \mathcal{D}^{\prime}(\Omega),\\
\left[ \mathfrak D_{ij}^{\alpha\beta}(t,x,\yply)\partial_{\beta}u_{j} - \mathfrak F_{ij}^{\alpha}(t,x,\yply)u_{j}+f_i^\alpha\right]\nu_{\alpha} & =  0 & &\text{on } \partial\Omega, \quad i=1,\ldots,m\\
u(0,x) & = u^{0} & &\text{in } \Omega.
\end{aligned}
\end{align}
Furthermore, the solution map 
\begin{equation} \label{solution_map}
S: \left(\yply,u^{0},f\right)\to u\text{, where }u\text{ is the solution of \eqref{eq:parab-linear},}
\end{equation}
 satisfies 
\[
\big \|S(\yply,u^{0},f) \big \|_{W(Q_{T})} \left[1- \epsilon K_0(\left\Vert \yply\right\Vert_{W(Q_{T})}) \right]
\leq \frac{1}{2} C_{T/\infty} \left(\| u^{0}\| _{H^{2}(\Omega)}+\left\Vert f\right\Vert _{ C([0,T];H^1\left(Q_{T}\right))\cap H^1(0,T;L^2(\Omega))}\right),
\]
where $C_{T/\infty}>0$ is given by \eqref{eq:ctinfty} and does not depend on $T$ if \textbf{(H)} holds. 
\end{lmm}
The proof of Lemma \ref{lem:313} is in Appendix \ref{sec:appendix}.  This first result has an immediate corollary.
\begin{crllr}
\label{cor:boundedsequence} For any $u^0$ and $\yply$ in $W(Q_T)$, suppose that
\[
\left[\mathfrak D_{ij}^{\alpha\beta}(t,x,\yply)\partial_{\beta}u^0_{j} - \mathfrak F_{ij}^{\alpha}(t,x,\yply)u^0_{j} \right] \cdot\nu=0 \quad \text{on }\partial\Omega,
\]
and 
\[
\epsilon\leq
 \frac{1}{2+2K_0 \left(C_{T/\infty} \| u^{0}\| _{H^{2}(\Omega)} \right)}, \qquad \|\yply\|_{W(Q_{T})} \leq Y_0,
\]
where $K_0$, $C_{T/\infty}$, and $Y_0$ are defined in \eqref{eq:defKo}, \eqref{eq:ctinfty}, and \eqref{eq:DefY0} respectively. Then 
\[
\big \|S(\yply,u^{0},0) \big \|_{W(Q_{T})}< Y_0.
\]
\end{crllr}
\begin{proof}
Since $K_0$ is a non decreasing function, we obtain
$$
\epsilon K_0\left(\|\yply\|_{W(Q_{T})} \right) \leq 
\frac{K_0 \left(Y_0\right) }
{2+2K_0 \left(Y_0\right)} <\frac{1}{2},
$$
hence \eqref{eq:epsilonbound} is satisfied.
Applying Lemma \ref{lem:313} with $f=0$, we obtain the announced estimate. 
\end{proof}
In a second step, we establish a contraction property.

\begin{lmm}
\label{lem:317} Given $\epsilon>0$, $u^{0}\in H^{2}(\Omega)$, and $\yply, \tilde \yply \in W(Q_{T})$, suppose that on $\partial\Omega$
\[
\left[\mathfrak D_{ij}^{\alpha\beta}(t,x,\yply)\partial_{\beta}u^0_{j} - \mathfrak F_{ij}^{\alpha}(t,x,\yply)u^0_{j} \right]\cdot\nu=0, \qquad 
\left[\mathfrak D_{ij}^{\alpha\beta}(t,x,\tilde\yply)\partial_{\beta}u^0_{j} - \mathfrak F_{ij}^{\alpha}(t,x,\tilde\yply)u^0_{j} \right] \cdot\nu=0.
\]
Suppose also that 
\[
\epsilon\leq \frac{1}{2[1+K_0 (Y_0 )]}, \qquad 
 \max(\|\yply\|_{W(Q_{T})},
\|\tilde\yply\|_{W(Q_{T})}) \leq Y_0,
\]
where $K_0$, $C_{T/\infty}$, and $Y_0$ are defined in \eqref{eq:defKo}, \eqref{eq:ctinfty}, and \eqref{eq:DefY0} respectively. Then we have 
\[
 \big \Vert S (\yply,u^{0},0 )-S(\tilde \yply,u^{0},0) \big \Vert_{W(Q_{T})}\leq\epsilon K_{1}(Y_0)\Vert \yply-\tilde \yply \Vert_{{{W(Q_{T})}}},
\]
with $K_1$ given by \eqref{eq:defK1}. 
\end{lmm}

\begin{proof}
Write $u=S(\yply,u^{0},0)$ and $\tilde u=S(\tilde \yply,u^{0},0)$. We have 
\[
u-\tilde u=\epsilon S\left(\yply,0,g\right)
\]
 where 
\begin{equation}
g_{i}^\alpha=a_{ij}^{\alpha\beta}(t,x)\big[\phi_{ij}^{\alpha\beta}(\yply)
-\phi_{ij}^{\alpha\beta}(\tilde \yply)\big]\partial_{\beta}\tilde u_j+b_{ij}^{\alpha}(t,x)\big[\psi_{ij}^{\alpha\beta}(\yply)-\psi_{ij}^{\alpha\beta}(\tilde \yply)\big]\tilde u_j.
\end{equation}
Noting that
\[
\big|\phi_{ij}^{\alpha\beta}(\yply)-\phi_{ij}^{\alpha\beta}(\tilde \yply)\big| \leq L_1(Y_0)  \big|\yply- \tilde \yply\big|,
\]
we find
\[
\max_{[0,T]}\|g\|_{L^2(\Omega)} \leq M L_1(Y_0) \|\yply- \tilde \yply \|_{W(Q_T)} \left\|\tilde u\right\|_{W(Q_T)}  \leq M L_1(Y_0) M_0 \|\yply- \tilde \yply \|_{W(Q_T)}.
\] 
Similarly, we can estimate the gradient as follows
\begin{multline*}
\left| \nabla g \right| \leq  M \left( L_1(Y_0)  \big|\yply- \tilde \yply\big|+ L_2(Y_0)  \big|\yply- \tilde \yply\big| \left| \nabla \yply\right|
+  L_1(Y_0)  \big|\nabla \yply- \nabla \tilde \yply\big|
\right)\left(  \left| \nabla \tilde u\right| + \left|\tilde u\right|\right) \\
 + M L_1(Y_0)  \big| \yply-  \tilde \yply\big| \left( \left| \nabla^2 \tilde u\right| 
+  \left| \nabla \tilde u\right| \right).
\end{multline*}
Therefore
\begin{multline*}
\|\nabla g\|_{L^2(\Omega)}
\leq M L_1(Y_0) 
\left[2  \|\yply- \tilde \yply \|_{L^\infty(Q_T)}  \| \tilde u  \|_{H^2(\Omega)} +  \|\yply- \tilde \yply \|_{L^4(\Omega)} 
\left( \| \nabla \tilde u \|_{L^4(\Omega)} + \| \tilde u \|_{L^4(\Omega)}\right)\right]
\\
+M L_2(Y_0)  \|\yply- \tilde \yply \|_{L^\infty(Q_T)} \left\|\nabla \yply \right \|_{L^4(\Omega)} 
 \left( \left\| \nabla \tilde u\right \|_{L^4(\Omega)} +\left\| \tilde u\right \|_{L^4(\Omega)}\right).
\end{multline*}
Thanks to the Ladyzhenskaya (or Gagliardo--Nirenberg) inequality, we obtain
\[
\max_{[0,T]}\|\nabla g\|_{L^2(\Omega)} \leq  C_S M \left(L_1(Y_0) Y_0 +L_2(Y_0) Y_0^2\right)  \|\yply- \tilde \yply \|_{W(Q_T)},
\] 
where $C_S^{1}$ is a product of Sobolev embedding constants, depending on $\Omega$ and $d$, namely 
\begin{equation}\label{eq:defCS1}
C_S^{1}=\max\left(1,C\left(H^2(\Omega)\hookrightarrow L^\infty(\Omega) \right)^3,C\left(H^2(\Omega)\hookrightarrow W^{1,4}(\Omega) \right)^3\right).  
\end{equation}
We now turn to the time derivative
\begin{multline*}
\left| \partial_t g \right| \leq  M \left( L_1(Y_0)  \big|\yply- \tilde \yply\big|+ L_2(Y_0)  \big|\yply- \tilde \yply\big| \left| \partial_t \yply\right|
+  L_1(Y_0)  \big|\partial_t \yply- \partial_t \tilde \yply\big|
\right)\left(  \left| \nabla \tilde u\right| + \left|\tilde u\right|\right) \\
 + M L_1(Y_0)  \big| \yply-  \tilde \yply\big| \left( \left| \nabla \partial_t \tilde u\right| 
+  \left| \partial_t \tilde u\right| \right).
\end{multline*}
Thus, using that $\partial_t \yply, \partial_t \tilde \yply \in L^4(Q_T)$ and $\partial_t \nabla \tilde u \in L^2(Q_T)$, we have
\[
\left\| \partial_t g \right\|_{L^2(Q_T)} \leq C^{2}_S M \left(L_1(Y_0) Y_0 +L_2(Y_0) Y_0^2\right)  \left\|\yply- \tilde \yply\right\|_{W(Q_T)},
\]
where $C_S^{2}$ is also a product of Sobolev embedding constants, depending on $\Omega$ and $d$, namely
\begin{equation}\label{eq:defCS2}
C_S^{2}=\max\left(C\left(H^1(\Omega)\hookrightarrow L^4(\Omega) \right)^2,1\right).  
\end{equation}

Finally, we apply Lemma \ref{lem:313} to obtain
\[
\| u-\tilde u \|_{W(Q_T)} \leq \epsilon Y_0 C_S M \left[L_1(Y_0) Y_0 +L_2(Y_0) Y_0^2\right]  \|\yply- \tilde \yply \|_{W(Q_T)},
\]
with 
\begin{equation}\label{eq:defCSK1}
C_S=C_S^{1}+C_S^{2}.  
\end{equation}
\end{proof}

We now turn to the proof of Proposition \ref{prop:315}. 
\begin{proof}[Proof of Proposition \ref{prop:315}]
Recall that 
\[
\epsilon_0 : R\to \min \left(\frac{1}{2+2K_0(R)},\frac{1}{ 1+ K_1 (R)}\right).
\]
where $K_0$ and $K_1$ are defined in \eqref{eq:defKo} and \eqref{eq:defK1}, respectively.

Given $u^{0} \in W(Q_T)$ we introduce the sequence $v_n$ given by $v_0=u^0$ and, for all $n\geq0$, 
\[
v_{n+1} = S(v_n,u^0,0),
\]
where $S$ is the solution map defined in \eqref{solution_map}. 
Note that the compatibility condition \eqref{eq:compatcondu_0} is satisfied at every step. Corollary~\ref{cor:boundedsequence} shows that $\| v_n\|_{W(Q_T)} \leq Y_0$ for each $n$.  Furthermore, thanks to Lemma \ref{lem:317},
\[
\|v_{n+2}-v_{n+1}\|_{W(Q_{T})}\leq \epsilon_0 K_1 (Y_0)  \|v_{n+1}-v_{n}\|_{W(Q_{T})} \leq \frac{K_1 (Y_0)}{1+K_1 (Y_0)}  \|v_{n+1}-v_{n}\|_{W(Q_{T})}.
\]
The sequence thus converges to a solution of \eqref{eq:1091}, thanks to the contraction mapping theorem.
\end{proof}

We now turn to the proof of the perturbation result in Theorem \ref{thm:297}. Consider the linearised system given by
\begin{align}\label{eq:parab-linear-1}
\begin{aligned}
	\partial_{t}\tilde u_{i}-\partial_{\alpha}\left[\widetilde {\mathfrak D}_{ij}^{\alpha\beta}(t,x,\tilde \yply)\partial_{\beta}\tilde u_{j} - \widetilde {\mathfrak F}_{ij}^{\alpha}(t,x,\tilde \yply)\tilde u_{j}\right] & =  \tilde{f}_i & &\text{in}\quad \mathcal{D}^{\prime}(\Omega),\\
\left[ \widetilde {\mathfrak D}_{ij}^{\alpha\beta}(t,x,\tilde \yply)\partial_{\beta}\tilde u_{j} - \widetilde {\mathfrak F}_{ij}^{\alpha}(t,x,\tilde \yply) \tilde u_{j}\right]\nu_{\alpha} & =  0 & &\text{on} \quad \partial\Omega,\\
\tilde u(0,\cdot) & = \tilde{u}^{0} & &\text{in} \quad \Omega,
\end{aligned}
\end{align}
Following the notation of Lemma \ref{lem:313} (see \eqref{solution_map}),
the solution operator associated to \eqref{eq:parab-linear-1} is denoted by $\tilde{S}(\tilde \yply,\tilde{u}^{0},\tilde{f})$.  

\begin{prpstn}
\label{prop:perturb-gap-linear}
Let $\yply,\tilde{\yply} \in W(Q_{T})$ be compactly supported in $\Omega$ for $t=0$ and write 
$$
Y_1 = C_{T/\infty} \max \left(\|\tilde{\yply}\|_{W(Q_{T})},\|\yply\|_{W(Q_{T})} \right).
$$
Assume $\epsilon < \epsilon_0 (Y_1)$, 
so that the solution operators $S$ and $\tilde{S}$ corresponding to \eqref{eq:parab-linear} and \eqref{eq:parab-linear-1}
respectively are well defined.  
For any $u^{0},{\tilde u}^{0}\in H^{2}(\Omega)$ with compact 
support in $\Omega$ there holds
\begin{align*}
	\Big\Vert \tilde{S}(\tilde{\yply},{\tilde u}^{0},0)&-S(\yply,u^{0},0 )\Big \Vert_{W(Q_{T})} \\ 
	&\leq C_{T / \infty} \| {\tilde u}^{0}-u^{0}\|_{H^{2}(\Omega)} +\epsilon K_1(Y_1) \|\tilde{\yply}-\yply\|_{W(Q_{T})} \\
& \quad +\epsilon K_2(Y_1) \left(\Vert (\tilde{a},\tilde{b})-(a,b)\Vert _{C^{1}([0, \infty)\times\mathbb{R}^{d})}
+\Vert (\tilde{\phi},\tilde{\psi})-(\phi,\psi)\Vert_{C^{1}\left(\overline{B_{Y_{1}}(0)}\right)}\right).
\end{align*} 
where $K_2$ depends on $L_0$, $L_1$, $\Omega$, $M$ and $C_{T / \infty}$ and is given by \eqref{eq:DefK2}. 
\end{prpstn}
\begin{proof}
We write
\begin{equation*}
\tilde{S}(\tilde{\yply},{\tilde u}^{0},0)-S(\yply,u^{0},0) = \tilde{S}(\tilde{\yply},{\tilde u}^{0},0)-S(\tilde{\yply},{\tilde u}^{0},0)+S(\tilde{\yply},{\tilde u}^{0},0) -S(\yply,{\tilde u}^{0},0) +S(\yply,{\tilde u}^{0},0)-S(\yply,u^{0},0).
\end{equation*}
Thanks to Lemma \ref{lem:313} and to the linearity of $S$ with respect to the initial data, we have
\[
\left \Vert S\left(\yply,{\tilde u}^{0},0\right)-S\left(\yply,u^{0},0\right)\right \Vert_{W(Q_{T})}\leq C_{T / \infty} \| u^{0}-{\tilde u}^{0}\|_{H^{2}(\Omega)}.
\]
Note that the compatibility condition is satisfied due to the compact support of $u^0$, ${\tilde u}^0$ and $\yply(t=0)$ in $\Omega$.  On the other hand, Lemma \ref{lem:317} shows that 
\[
\left \Vert S(\tilde{\yply},{\tilde u}^{0},0)-S(\yply,{\tilde u}^{0},0)\right \Vert_{W(Q_{T})}
\leq
\epsilon K_{1}(Y_{1}) \Vert \tilde{\yply}- {\yply} \Vert _{W (Q_{T})}.
\]
We write 
\[
\tilde{S}(\tilde{\yply},{\tilde u}^{0},0)-S(\tilde{\yply},{\tilde u}^{0},0)=\epsilon\tilde{S}(\tilde{\yply},0,\tilde{g}),
\]
where $\tilde{g}$ is given by 
\begin{align*}
\tilde{g}_i^\alpha & =\epsilon^{-1}\left[\left(\tilde{\mathfrak D}_{ij}^{\alpha\beta}-\mathfrak D_{ij}^{\alpha\beta}\right)(t,x,\tilde{\yply})\partial_{\beta}\tilde  u_{j} - \left(\tilde{\mathfrak F}_{ij}^{\alpha}- \mathfrak F_{ij}^{\alpha}\right)(t,x,\tilde{\yply})\tilde{u}^{j}\right]\\
 & =\left(\tilde{a}_{ij}^{\alpha\beta}\tilde{\phi}_{ij}^{\alpha\beta}-a_{ij}^{\alpha\beta}\phi_{ij}^{\alpha\beta}\right)(t,x,\tilde{\yply})\partial_{\beta}\tilde u_{j}+\left(\tilde{b}_{i}^{\alpha\beta}\tilde \psi_{ij}^{\alpha}-b_{ij}^{\alpha}\psi_{ij}^{\alpha}\right)(t,x,\tilde{\yply})\tilde{u}^{j},
\end{align*}
and $\tilde{u}=S(\tilde{\yply},{\tilde u}^{0},0)$. In other words, $\tilde g$ is of the form
\[
\tilde g = \left[ (\tilde{a} - a) \tilde{\phi} + a (\tilde\phi -\phi)\right]\nabla \tilde u + \left[ (\tilde{b} - b) \tilde \psi + b (\tilde\psi -\psi)\right] \tilde u,
\]
and thus we are in a setting similar to that of the proof of Lemma \ref{lem:317}. In particular we have
\begin{multline*}
 (|\nabla g|+|g|) \leq 
\left( \big\Vert (\tilde{a},\tilde{b})-(a,b)\big\Vert _{C^{1}([0,\infty)\times\mathbb{R}^{d})} L_0(Y_1)  +  M \max_{B_{Y_{1}}(0)} \big| (\tilde{\phi},\tilde{\psi})-(\phi,\psi)\big|   
\right) (|\nabla^2 \tilde u| +2|\nabla \tilde u| + |\tilde u|) \\
 + \left(
                 \big|(\tilde{a},\tilde{b})-(a,b)\big| L_1(Y_1)  
  + M \max_{B_{Y_{1}}(0)} \big| (\tilde{D\phi},\tilde{D\psi})-(D\phi,D\psi)\big| 
  \right)
  | \nabla \tilde \yply |  (|\nabla \tilde u| + |\tilde u|)
\end{multline*}
As in the proof of Lemma~\ref{lem:317}, using Gagliardo--Nirenberg's inequality to bound the last term, we find 
\begin{multline*}
\frac{1}{C_S} \max_{[0,T]} \| \tilde g \|_{H^1(\Omega)} \\
\leq 
\big\Vert (\tilde{a},\tilde{b})-(a,b)\big\Vert _{C^{1}([0,\infty)\times\mathbb{R}^{d})} \left[2 L_0(Y_1)Y_1 +  L_1(Y_1)Y_1^2 \right] 
+  M  \big\Vert (\tilde{\phi},\tilde{\psi})-(\phi,\psi)\big\Vert _{C^{1}(\overline{B_{Y_{1}}(0)})} (2Y_1+Y_1^2),  
\end{multline*}
where $C_S$ is given by \eqref{eq:defCSK1}.
Finally, we bound $\partial_t g $ to show that $\tilde g\in C([0,T];H^1(\Omega))\cap H^1(0,T;L^2(\Omega))$ in the same way, namely 
\begin{multline*}
\frac{1}{C_S}   \| \partial_t \tilde g \|_{L^2(Q_T)} \\
\leq 
\big\Vert (\tilde{a},\tilde{b})-(a,b)\big\Vert _{C^{1}([0,\infty)\times\mathbb{R}^{d})} \left[ L_0(Y_1)Y_1 +  L_1(Y_1)Y_1^2 \right] 
+ M  \big\Vert (\tilde{\phi},\tilde{\psi})-(\phi,\psi)\big\Vert _{C^{1}(\overline{B_{Y_{1}}(0)})} (  Y_1+Y_1^2)  .
\end{multline*}
Because of the compact support of $u^0$, ${\tilde u}^0$, $\tilde{\yply}(t=0)$ and $\yply(t=0)$ in $\Omega$, we  can conclude thanks to Lemma~\ref{lem:313} that 
\[
  \big \| \tilde{S}(\tilde{\yply},{\tilde u}^{0},0 )-S (\tilde{\yply},{\tilde u}^{0},0 ) \big \|_{W(Q_T)} 
  \leq \epsilon  K_2 (Y_1) \left(\big\Vert (\tilde{a},\tilde{b})-(a,b)\big\Vert _{C^{1}([0,\infty)\times\mathbb{R}^{d})} +  
  \big\Vert (\tilde{\phi},\tilde{\psi})-(\phi,\psi)\big\Vert _{C^{1}(\overline{B_{Y_{1}}(0)})}   \right).
\]
\end{proof}

\begin{proof}[Proof of Theorem~\ref{thm:297}]
As in the proof of Proposition \ref{prop:315}, the 
sequences $v_{n+1}=S\left(v_{n},u^{0},0\right)$ and $\tilde{v}_{n+1}=\tilde{S}\left(\tilde{v}_{n},{\tilde u}^{0},0\right)$ for
all $n\geq 1$, with $v_{0}=u^0$ and $\tilde{v}_{0}= \tilde u^0$,  converge to $u$ and $\tilde{u}$, respectively as $n\to\infty$. Thanks to Proposition \ref{prop:perturb-gap-linear}
we have 
\begin{multline*}
	\Vert \tilde{v}_{n+1}-v_{n+1}  \Vert_{W(Q_{T})} \leq C_{T / \infty} \| {\tilde u}^{0}-u^{0}\|_{H^{2}(\Omega)} 
	+\epsilon K_1(Y_1) \|\tilde{v}_{n}- v_n\|_{W(Q_{T})} \\
 +\epsilon K_2(Y_1) \left(\Vert (\tilde{a},\tilde{b})-(a,b)\Vert _{C^{1}([0, \infty)\times\mathbb{R}^{d})}+\Vert (\tilde{\phi},\tilde{\psi})-(\phi,\psi)\Vert_{C^{1}(\overline{B_{Y_{1}}(0)})}\right).
\end{multline*}
Passing to the limit as $n\to\infty$,  we obtain
\begin{multline*}
	\Vert \tilde{u}-u  \Vert_{W(Q_{T})}  \leq [1+ K_1(Y_1)]C_{T / \infty} \| {\tilde u}^{0}-u^{0}\|_{H^{2}(\Omega)} \\
 +\epsilon [1+K_1(Y_1)]K_2(Y_1) \left(\Vert (\tilde{a},\tilde{b})-(a,b)\Vert _{C^{1}([0, \infty)\times\mathbb{R}^{d})}+\Vert (\tilde{\phi},\tilde{\psi})-(\phi,\psi)\Vert_{C^{1}(\overline{B_{Y_{1}}(0)})}\right),
 \end{multline*}
 as required. 
\end{proof}

\section{Numerical simulations} \label{sec:numerics}

In this section, we present numerical simulations for the cross-diffusion systems described as Examples 1,2 and 3 in the introduction.  We consider these examples when the physical dimension is $d=2$, but with initial data and potentials $V_i$ varying in one direction such that the solutions of \eqref{eq:cross-diff_general} can be represented as one-dimensional. We solve \eqref{eq:cross-diff_general} in the domain $\Omega = (-1/2, 1/2)$ 
using a second-order accurate finite-difference scheme in space and the method of lines with the inbuilt Matlab ode solver \texttt{ode15s} in time. We use an equidistant mesh of size $|\Omega|/J$, with nodes $x_n = -1 + n \Delta x$, $0 \le n \le J$. 
The fluxes are evaluated at the nodes $x_n$ to ensure the no-flux conditions are imposed accurately, while the solutions $u_i$ are computed at the midpoints $x_{n+1/2}$. The unknowns are $u_{i,n}(t) \approx u_i(x_n,t)$, $i = 1,2$. 
The discretisation of the spatial derivatives is done in the spirit of the positivity-preserving scheme proposed in \cite{Zhornitskaya:2000hr}. For example, the terms of the form $u_i \nabla u_j$ are discretised as 
\[
\left (u_i \frac{\partial u_j}{\partial x} \right)(x_{n+ 1/2}) \approx \left ( \frac{ 2 u_{i,n+1} u_{i,n}}{u_{i,n+1} + u_{i,n}} \right) \left( \frac{ u_{j,n+1} - u_{j,n}}{\Delta x} \right).
\]

We begin with a simulation of the model in Example \ref{exa:model1} for fixed $\epsilon$ to demonstrate a typical evolution of a cross-diffusion system. 
The value of the parameters used in the numerical implementation are given below. 
Recall that the model describes two species of hard sphere particles in $\mathbb R^d$, $d=2, 3$, possibly with different numbers $N_i$, diffusions $D_i$, and diameters $\varepsilon_i$. 
The coefficients in \eqref{model1} are given by
\begin{equation} \label{coef_model1}
	a_i = \frac{2\pi}{d}(d-1) \bar N_i \bar \varepsilon_i^d, \quad b_i = \frac{2\pi}{d}\frac{[(d-1)D_i + d D_j]}{D_i+D_j} \bar N_i, \quad c_i = \frac{2\pi}{d}\frac{D_i}{D_i+D_j} \bar N_j,
\end{equation}
for $i,j = 1, 2$ ($j \ne i$), where $\bar N_i = N_i/(N_1 + N_2)$, $\varepsilon = (\varepsilon_1 + \varepsilon_2)/2$, $\bar \varepsilon_i = \varepsilon_i/\varepsilon$. In particular $\bar N_1 + \bar N_2 = 1$ and $\bar \varepsilon_1 + \bar \varepsilon_2 = 2$. The small parameter $\epsilon$  is then defined as
\begin{equation}
	\epsilon = (N_1 + N_2) \varepsilon^d/|\Omega|. 
\end{equation} 

For our first example, we choose $N_1 = N_2 = 100$, $D_1 = D_2 = 1$, $\varepsilon_1 = \varepsilon_2 = 0.0354$, $d=2$. This gives the value $\epsilon = 0.25$. We set initial data $u_{1,0} = C\exp(-80(x+0.2)^2)$, where $C$ is the normalisation constant, and $u_{2,0}= 1$, and external potentials $V_1(x) = 1-\exp(-120 x^2)$ and $V_2(x) = 0$. We run the time-dependent simulation until $T = 1$ and plot the results in Figure \ref{fig:fig1}. We observe the evolution of $u_1$ towards a non-trivial steady state, governed by $V_1$, while $u_2$ diffuses away from the centre (despite having no external potential) due to the cross-species interaction. 
\def \scc {0.8}
\def \scl {1.0}
\begin{figure}
\unitlength=1cm
\begin{center}
\vspace{3mm}
\psfrag{x}[][][\scl]{$x$} \psfrag{t}[][][\scl]{$t$} \psfrag{u1}[][][\scl]{$ u_1$} \psfrag{u2}[][][\scl]{$u_2$} \psfrag{data1}[][][\scc]{$t_0$} \psfrag{t}[b][][\scl]{$t$}
\psfrag{data2}[][][\scc]{$t_1$}  \psfrag{data3}[][][\scc]{$t_2$}  \psfrag{data4}[][][\scc]{$t_3$}  \psfrag{data5}[][][\scc]{$t_4$} 
	\includegraphics[height = .38\textwidth]{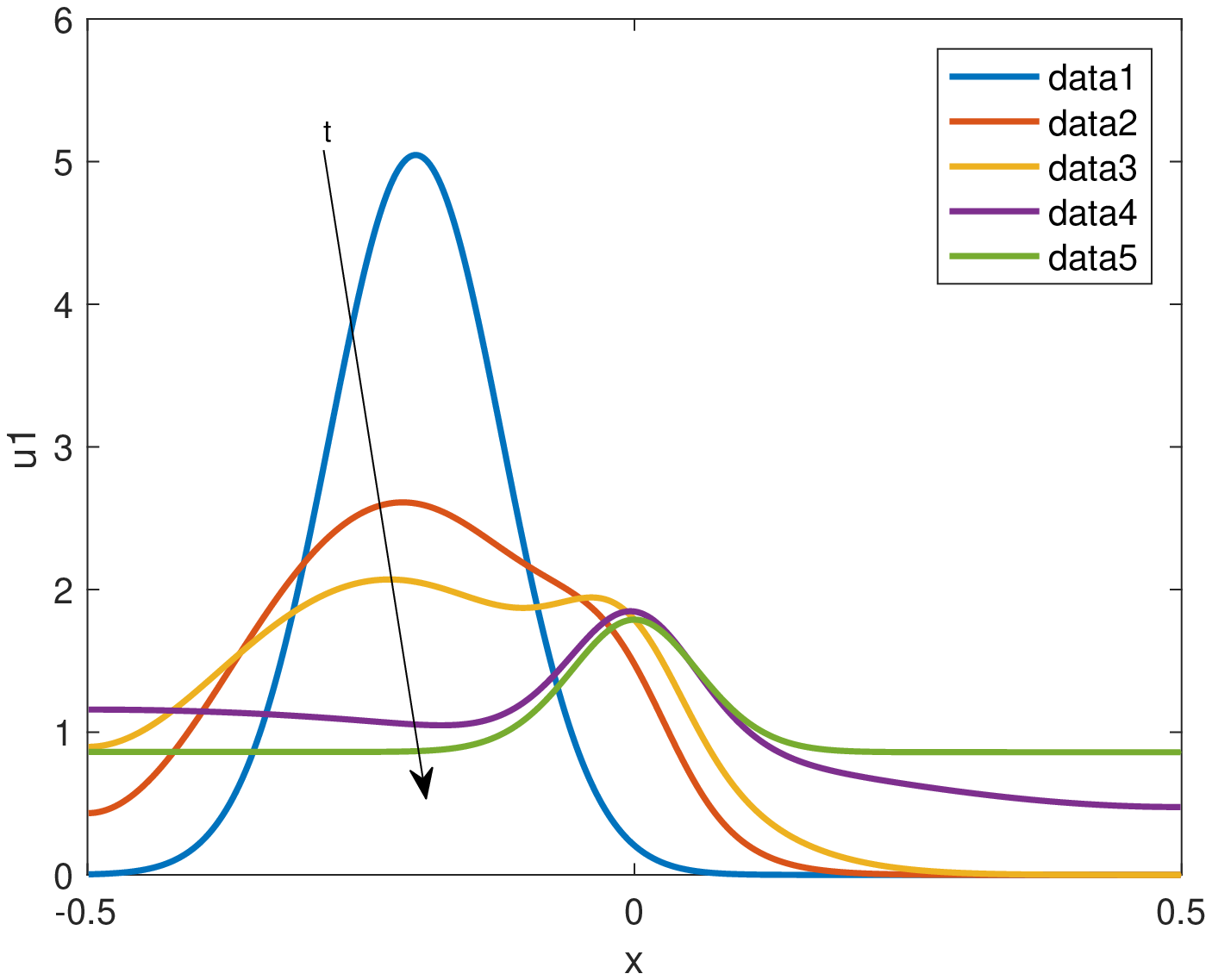}\qquad
	\includegraphics[height = .38\textwidth]{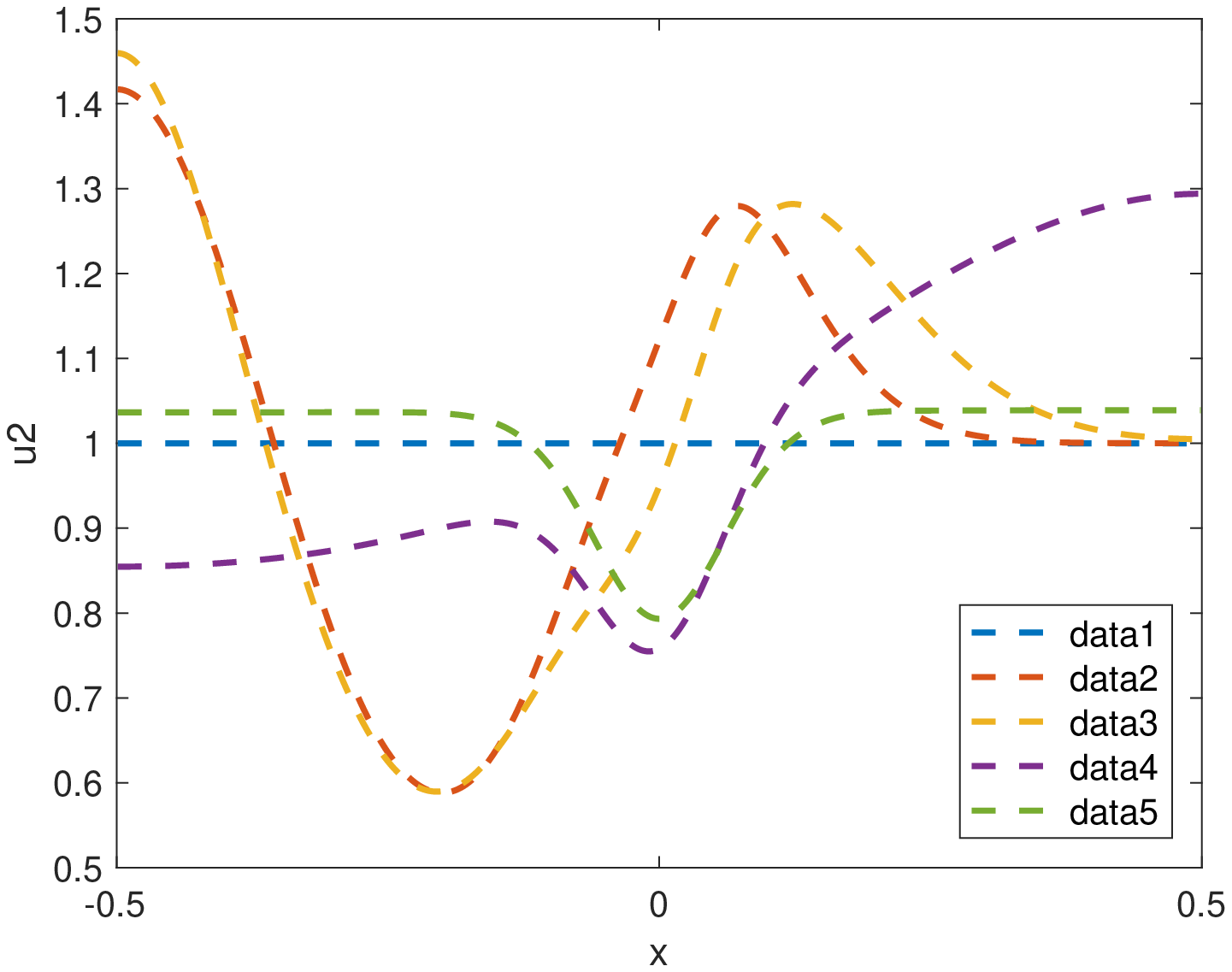}
\caption{Time-dependent simulation of the model \eqref{model1} in Example \ref{exa:model1}. Time-evolution of the population densities $u_1$ (left) and $u_2$ (right) with  initial data $u_1^0 = C\exp(-80(x+0.2)^2)$, where $C$ is the normalisation constant, and $u_2^0= 1$, and final time $T = 1$ (times shown $t_0=0, t_1 =0.005, t_2 = 0.01, t_3 = 0.1, t_4 = T$). The external potentials are $V_1(x) = 1-\exp(-120x^2)$ and $V_2(x) = 0$ and volume fraction parameter $\epsilon= 0.25$. The other parameters are: $\bar N_i = 1/2$, $\bar \varepsilon_i = 1$, $d=2$, $D_i = 1$, $J = 500$.}
\label{fig:fig1}
  \end{center}
\end{figure}

To show the dependence of the solutions of \eqref{model1} with the small parameter $\epsilon$, in Figure \ref{fig:fig2} we plot the steady state solution $u^\infty$ for three values of the occupied volume $\epsilon$, namely $\epsilon = 0, 0.125, 0.25$. This is obtained by running the time-dependent solver for long times; we found $T = 20$ to be sufficient. Convergence to a unique steady state is guaranteed by the results in \cite{Bruna:2016cm} and our time-independent estimates. We observe the effects of $\epsilon$: for $\epsilon =0$ (no interactions), $u_2 = 1$ is already the steady state solution. As we increase $\epsilon$, the maximum of $u_1^\infty$ goes down, as not so many particles can fit where the potential is minimised, and a minimum in $u_2^\infty$ appears where $u_1^\infty$ has its maximum, showing that particles from species 2 are pushed out driven by gradients in $u_1$. 
\def \scc {0.8}
\def \scl {1.0}
\begin{figure}
\unitlength=1cm
\begin{center}
\vspace{3mm}
\psfrag{x}[][][\scl]{$x$} \psfrag{t}[][][\scl]{$t$} \psfrag{u}[b][][\scl]{$u^\infty$} \psfrag{de}[][][\scl]{$\epsilon$} \psfrag{d}[][][\scl]{$\epsilon$} 
	\includegraphics[width = .6\textwidth]{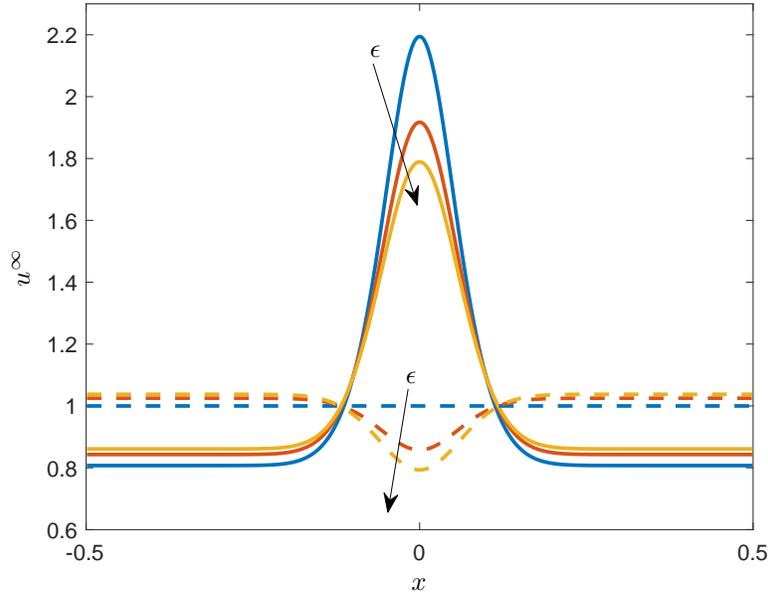}  
\caption{Steady state solutions $u_1^\infty$ (solid lines) and $u_2^\infty$ (dashed lines) of the model \eqref{model1} in Example \ref{exa:model1} for different values of $\epsilon$, $\epsilon = 0, 0.125, 0.25$ (arrows show the direction of increasing $\epsilon$). The other parameters are given in Figure \ref{fig:fig1}.}
\label{fig:fig2}
  \end{center}
\end{figure}

In the next simulation, we want to test the behaviour of the system in Example \ref{exa:model1} as the perturbation in $\epsilon$ increases. To make the calculation of the bounds simpler, we assume that $\bar \varepsilon_i = 1$, $\bar N_i = 1/2$, $d=2$, and that the two components of the solution coincide at at least one point, that is, $u_1=u_2=u^*$ for some $u^*>0$. We choose the initial data shown in Figure \ref{fig:fig3}(a) and $V_i = 0$, so that $u^* = \max_x u^{0} \approx 1.333$. We have already introduced a bound $\epsilon_0$ in \eqref{eq:defepsilonzero}, ensuring that the existence result in Theorem \ref{thm:297} holds. The expression of $\epsilon_0$ is found in the proof of Lemma \ref{lem:313} for a general system, but it can be improved for the specific system at hand. However, in this section we will use another bound, which we denote by $\epsilon^*$, that ensures ellipticity of the diffusion matrix \eqref{eq:diff_matrix}. This is in fact the practical bound required to obtain meaningful numerical results, and it is in general less restrictive than $\epsilon_0$.

\begin{lmm}[Ellipticity bound]\label{lem:second_bound}
The following condition is necessary to ensure coercivity of the diffusive term. Suppose that the solution of \eqref{eq:cross-diff_general} with matrices \eqref{model1} satisfies $\max_{Q_T} |u| = u^* >0$, $d=2$, and one of the following cases apply:
\begin{compactenum}[(i)]
\item Different diffusivities: $\bar \varepsilon_i = 1$, $\bar N_i = 1/2$,  and $\theta = (D_1-D_2)^2/4D_1 D_2 \ge 0$.
\item Different particle sizes: $D_i = 1$, $\bar N_i = 1/2$, $\bar \varepsilon_2 = 2- \bar \varepsilon_1$, and $\theta = 1 - 2 \bar \varepsilon_1  + \bar \varepsilon_1^2 \ge 0$.
\item Different particle numbers: $D_i = 1$, $\bar \varepsilon_i = 1$, $\bar N_2 = 1- \bar N_1$, and $\theta = 9 (1/4 - \bar N_1 + \bar N_1^2) \ge 0$.
\end{compactenum}
Then the symmetrised version of the diffusion matrix \eqref{eq:diff_matrix} is non-degenerate provided that
\[
\epsilon \le \epsilon^* = \frac{1+\sqrt{9+4 \theta}}{2+\theta} (\pi u^*)^{-1},
\]
where $\theta$ takes the values specified above. The bound is sharp in the case that both components $u_1$ and $u_2$ attain $u^*$ at the same point.
\end{lmm}

\begin{proof}
Recall that the diffusion matrix of Example \ref{exa:model1} is
\begin{equation*}
\mathfrak D(u)=\begin{pmatrix}D_1(1+\epsilon a_1 u_1-\epsilon  c_1 u_2) & \epsilon D_1  b_1 u_1 \\
\epsilon D_2 b_2 u_2 & D_2(1+\epsilon  a_2 u_2 -\epsilon  c_2 u_1)
\end{pmatrix}. \label{eq:diff_matrix0}
\end{equation*}

From the numerical point of view, a realistic bound can be obtained imposing that the symmetrised diffusion matrix does not degenerate.
We consider the case (i), that is, $\bar \varepsilon_i = 1$, $\bar N_i = 1/2$. Suppose that both components $u_1,u_2$ attain the same maximum at the same point, $u_1 = u_2 = u^*$. 
We have
\begin{align*}
\det(\mathrm{Sym}(\mathfrak D))  =\mathrm{det}(\mathfrak D)-\left(\frac{\mathfrak D_{12}- \mathfrak D_{21}}{2}\right)^{2} 
= D_1 D_2 \left[ 1+ \frac{1}{2} \epsilon \pi u^*- \frac{1}{4} (\epsilon \pi u^*)^2 ( 2 + \theta) \right],
\end{align*}
where  $\theta = (D_1-D_2)^2/(4D_1 D_2) \ge 0$. Imposing that $\det(\mathrm{Sym}(D)) = 0$ leads to
\[
 \epsilon \pi u^* = \frac{1 + \sqrt{9 + 4 \theta}}{2 + \theta}
\]
as required. The other cases, as well as the non-sharp cases when, for instance, $u_1 < u_2 = u^*$, follow in a similar way. 
\end{proof}

To test the upper bounds on $\epsilon$, in the next example we run a simulation of model \eqref{model1} for increasing values of $\epsilon$. 
We expect the norm $\| u \|_{W(Q_T)}$ to increase suddenly for values 
$\epsilon > \epsilon^*$. 
In the example we consider, $\epsilon^* = 2/(\pi u^*) \approx 0.4776 $, and $\epsilon^* = 2\tilde \epsilon_0\gg \epsilon_0 \approx 2.57\times 10^{-5}$.   
In the simulations, we approximate the norm in $W_2(Q_T)$ as follows. Let $u_i(n,k)$ denote the finite-difference approximation of $u_i(x_n, t_k)$, 
where $x_n$ and $t_k$ are $J$ and $M$ equally spaced nodes in $\Omega = [-1/2, 1/2]$ and $[0, T]$ respectively, $x_n = -1/2 + n \Delta x$, $\Delta x = 1/J$ and $t_k = 0 + k \Delta t$, $\Delta t = T/M$. Then 
\begin{align} \label{discrete_norm}
\begin{aligned}
	\| u \|_{W(Q_T)} \approx &\sqrt{ \Delta x \Delta t \sum_{n,k} \left[ u_{1xx}^2(n,k) + u_{2xx}^2(n,k) + u_{1t}^2(n,k) + u_{2t}^2(n,k) \right] } \\
	& + \max_k \sqrt{ \Delta x\sum_n \left[ u^2(n,k) + u_{2}^2(n,k) + u_{1x}^2(n,k) + u_{2x}^2(n,k) \right] },
\end{aligned}
\end{align}
where $u_{ixx}(n,k) = [u_i(n+1,k) + u_i(n-1,k)-2 u_i(n,k)]/\Delta x^2$, $u_{ix}(n,k) = [u_i(n+1,k) - u_i(n-1,k)]/ (2\Delta x)$ and $u_{it}(n,k) = [u_i(n,t+1)- u_i(n,t)]/\Delta t$.  
We choose initial data $u^{0}$ such that the two components attain the same maximum $u^*$ in regions that overlap (see Figure \ref{fig:fig3}(a)), 
and zero external potentials $V_i$ so that we can ensure that the maximum of $u^{0}$ is also the global maximum. We consider the symmetric case when diffusivities, 
particle numbers and sizes are equal, $\bar \varepsilon_1 = \bar \varepsilon_2 = 1$, $\bar N_1 = \bar N_2 = 1/2$, $D_1 = D_2$, so that $\theta \equiv 0$ 
and $\epsilon^* = 2/(\pi u^*) = $ from Lemma \ref{lem:second_bound}. 
We observe that the norm $\| u \|_{W(Q_T)}$ blows up as expected for $\epsilon \ge 0.5$, when the determinant of the symmetrised diffusion matrix is negative.

\def \scc {0.6}
\def \scl {.8}
\begin{figure}
\unitlength=1cm
\begin{center}
\vspace{3mm}
\psfrag{a}[][][\scl]{ (a)} \psfrag{b}[][][\scl]{ (b)} \psfrag{c}[][][\scl]{ (c)}
\psfrag{x}[][][\scl]{$x$} \psfrag{t}[][][\scl]{$t$} \psfrag{u}[][][\scl]{$u^{0}$} \psfrag{det}[][][\scl]{$\det (\mathrm{Sym} ( \mathfrak D))$} \psfrag{norm}[][][\scl]{$\| u \|_{W(Q_T)}$} \psfrag{de}[][][\scl]{$\epsilon$}
	\includegraphics[width = .4\textwidth]{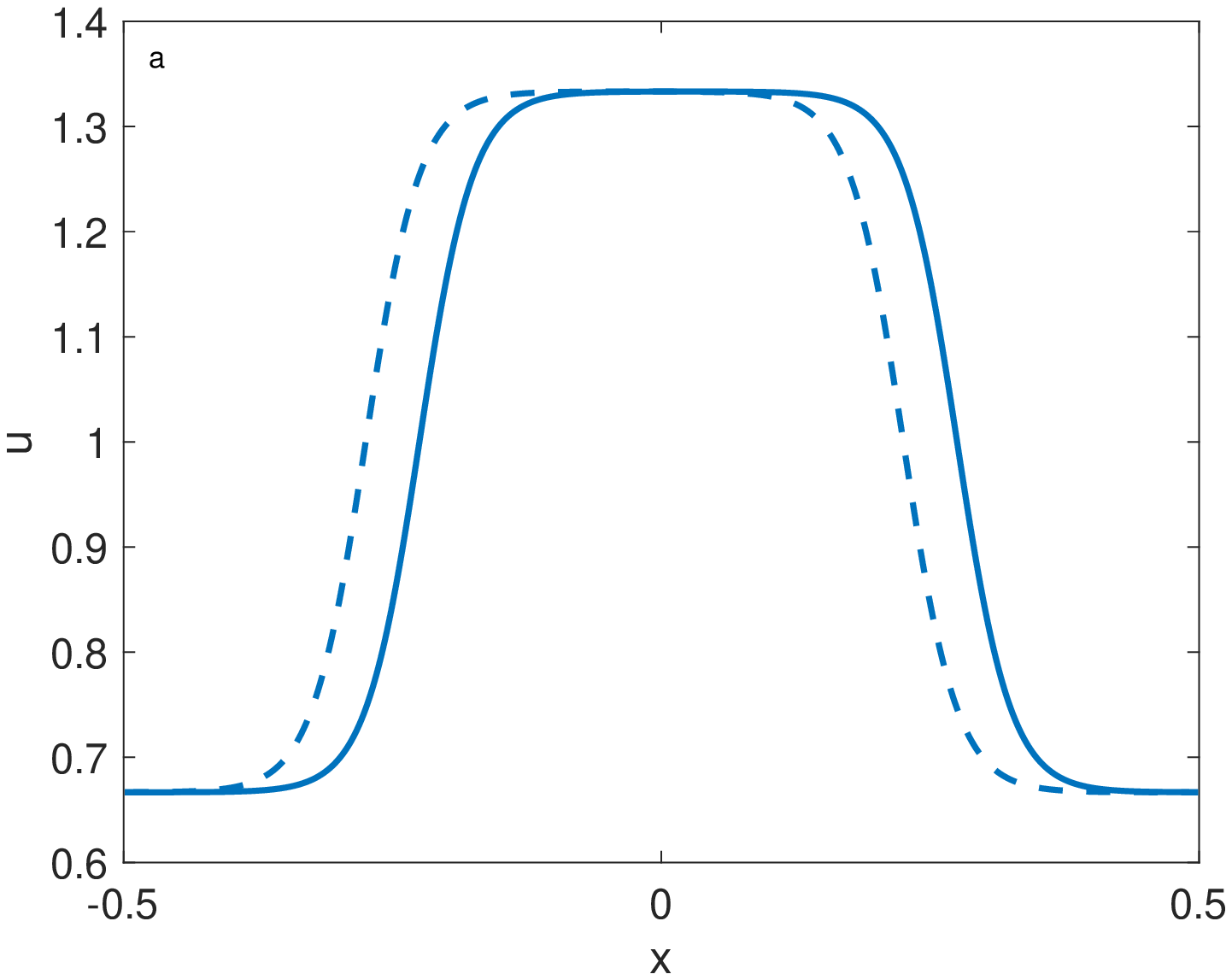} \\ \vspace{.3cm}
	\includegraphics[width = .4\textwidth]{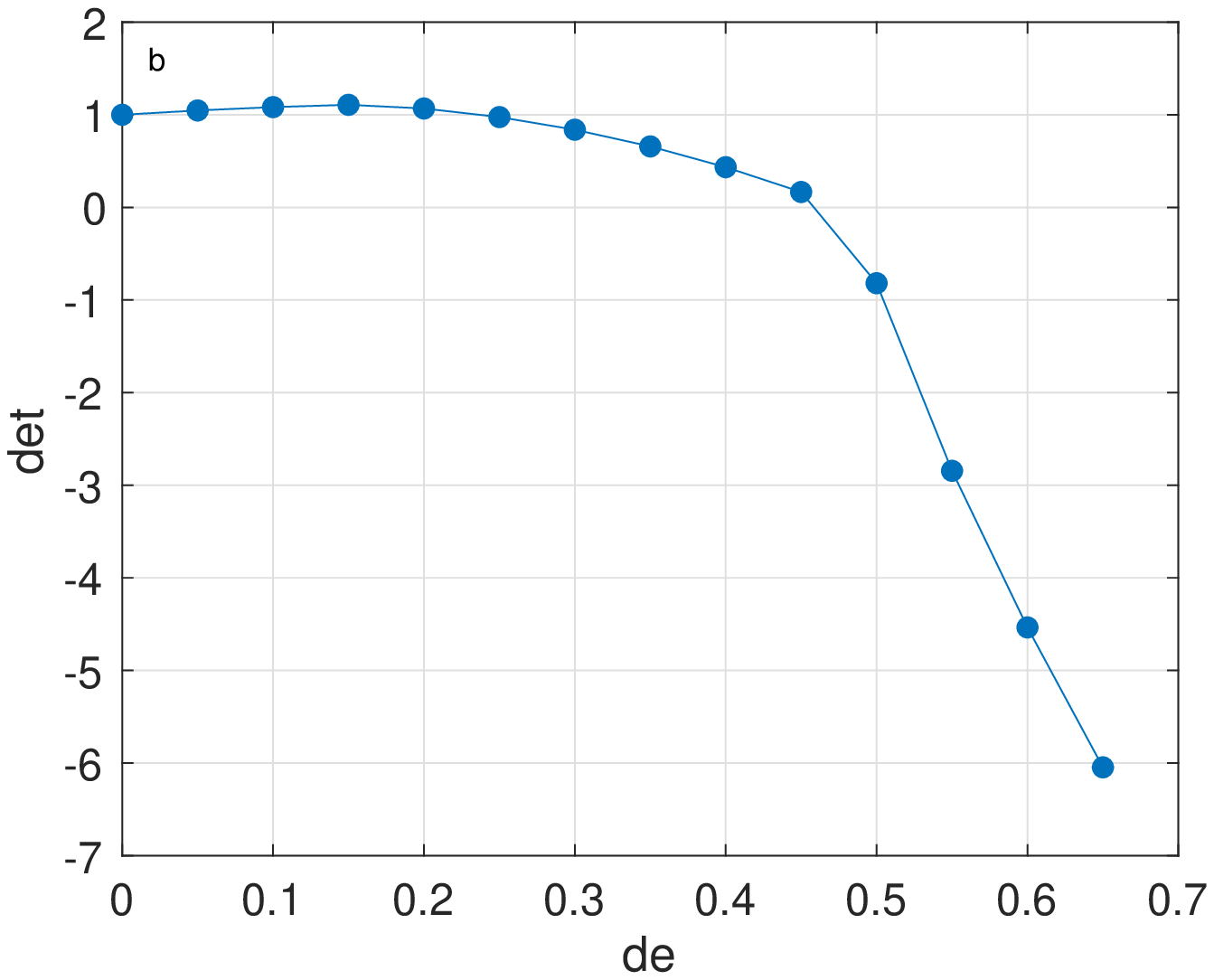} \quad 
	\includegraphics[width = .39\textwidth]{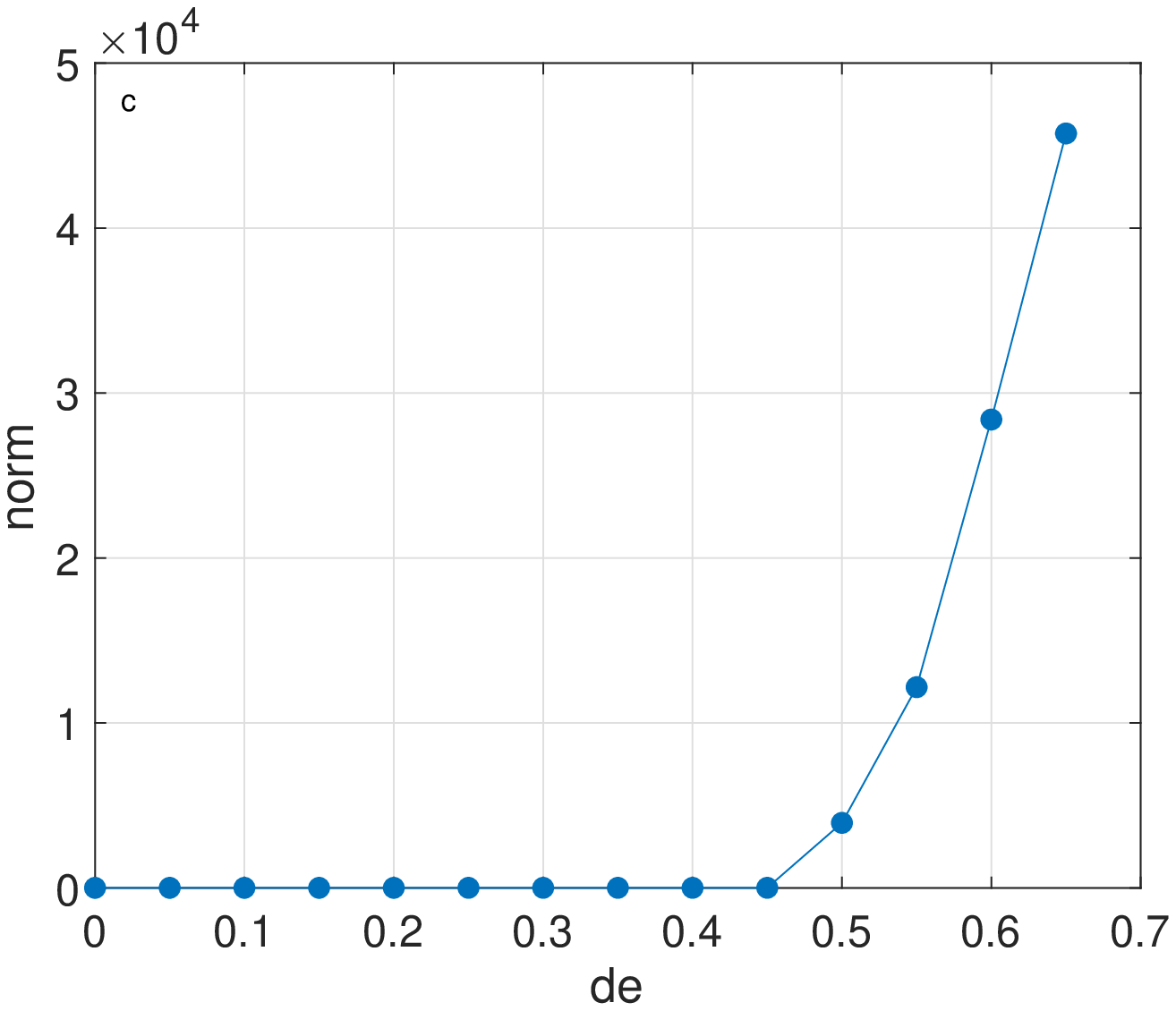}  
\caption{Simulation of model \eqref{model1} for increasing values of $\epsilon$. 
(a) Initial data: $u_1^0 = 1 + 0.5\tanh[10(2x+a-b)] + 0.5\tanh[-10(2x-a-b)]$ (solid line) and $u_2^0 =1 + 0.5\tanh[10(2x+a+b)] + 0.5\tanh[-10(2x-a+b)]$ (dashed line) with $a = 0.5$ and $b= 0.05$. 
(b)  Determinant of the symmetrised diffusion matrix \eqref{eq:diff_matrix} as a function of $\epsilon$.
(c) Norm $\| u \|_{W(Q_T)}$ computed using \eqref{discrete_norm} as a function of $\epsilon$. 
Parameters used: $\bar N_i = 1/2$, $\bar \varepsilon_i = D_i = 1$, $V_i = 0$, $d=2$, and $T = 0.1$, $J = 500$ and $M = 100$. }
\label{fig:fig3}
  \end{center}
\end{figure}

Our second set of simulations relates to stability under perturbations of the matrices $\mathfrak D$ and $\mathfrak F$. 
We compare the solutions of Example \ref{exa:model2} and  Example \ref{exa:model1}, and the solution of Example \ref{exa:model1} 
and the gradient-flow solution of Example \ref{ex:model3}. In the first case, the perturbation or differences between the models 
are at order $\epsilon$, whereas in the second case the differences are at order $\epsilon^2$. We would like to test the theoretical 
predictions of our analysis, namely, that we can control the difference between the solutions of the models in Examples \ref{exa:model2}, 
\ref{exa:model1} and \ref{ex:model3} by the difference in their diffusion and drift matrices. 

We denote by $\mathfrak D$ and $\mathfrak F$ the matrices of Example \ref{exa:model2}, and by $\widetilde {\mathfrak D}$ and $\widetilde {\mathfrak F}$ 
those of Example \ref{exa:model1}. The difference between the models is
\begin{subequations} \label{diff_models12}
\begin{equation}
\widetilde {\mathfrak D} - \mathfrak D= \epsilon \begin{pmatrix}D_1[ a_1 u_1 + u_2(\bar N_2 - c_1 )] &  D_1   u_1 (b_1 -   \bar N_2 ) \\
D_2 u_2 (b_2 - \bar N_1)  & D_2[  a_2 u_2 +  u_1 ( \bar N_1 - c_2)]
\end{pmatrix}, \label{eq:diff_modelsmatrix2}
\end{equation}
and
\begin{equation}\label{diffmodels_drift}
\begin{aligned}
\widetilde {\mathfrak F} - \mathfrak F &=  \epsilon \begin{pmatrix} -u_1 \nabla V_1 \bar N_1 &  u_1 [ (c_1 - \bar N_2) \nabla V_1 - c_1 \nabla V_2] \\
 u_2 [ (c_2 - \bar N_1) \nabla  V_2 - c_2 \nabla V_1 ] & - u_2 \nabla V_2  \bar N_2
\end{pmatrix} \\
&=  \epsilon \begin{pmatrix} - \nabla V_1 \bar N_1 &  [ (c_1 - \bar N_2) \nabla V_1 - c_1 \nabla V_2] \\
 [ (c_2 - \bar N_1) \nabla  V_2 - c_2 \nabla V_1 ] & - \nabla V_2  \bar N_2
\end{pmatrix} \circ \begin{pmatrix} u_1 &  u_1  \\
 u_2  &  u_2 
\end{pmatrix},
\end{aligned}
\end{equation}
In the second line, we rewrite the difference as two matrices, one dependent on $x$ and the 
other on $u$ (as required in our analysis), where $\circ$ denotes the Hadamard or entry-wise product of matrices. 
\end{subequations}

The difference between the model  \eqref{model1} in Example \ref{exa:model1} and the gradient-flow model \eqref{grad_flow_general} in 
Example \ref{ex:model3} is the order $\epsilon^2$ term $G$  (see \eqref{gradflow_generalasy}),  given by
\begin{equation}
	G = (\theta_1 \nabla u_1- \theta_2  \nabla u_2 )u_1 u_2 \begin{pmatrix}
		-D_1\\
		D_2
	\end{pmatrix},
\end{equation}
where  $\theta_1  =  a_1 c_1  - a_{12}  c_2$, $\theta_2  =a_2 c_2-a_{12}c_1$, and $a_{12}= (d-1)(c_1+c_2)$. Therefore, both models have  
the same drift matrices and their difference is contained in their respective diffusion  matrices. If we denote by $\hat {\mathfrak D}$ the  
diffusion matrix of model  \eqref{grad_flow_general},  then $(\hat {\mathfrak D}  - \widetilde {\mathfrak D}) \nabla u = \epsilon ^2 G$ (see \eqref{gradflow_generalasy}), that is,
\begin{equation}
	\hat {\mathfrak D} - \widetilde {\mathfrak D}  = \epsilon^2  u_1  u_2 \begin{pmatrix}
		-D_1\theta_1 & D_1\theta_2  \\
		D_2  \theta_1 & -D_2 \theta_2
	\end{pmatrix}.
\end{equation}

To test our stability results, we next compare the solutions of the models above in a simulation with initial data as in Figure \ref{fig:fig3}(a), 
equal particle numbers $\bar N_i =  1/2$, equal particle sizes $\bar\varepsilon_i= 1$ (since the lattice-based model in Example \ref{exa:model2} only 
admits equal sizes), and $D_1=1.5$, $D_2 = 1$. We plot the results in Figure \ref{fig:fig5}, using the potentials $V_1(x) = 1-\exp(-120x^2)$ and $V_2(x) = 0$ as in Figure \ref{fig:fig2}. 
As expected, the stability between models in Examples \ref{exa:model2} and \ref{exa:model1} is of order  $\epsilon$, whereas the difference  
between the solutions of models in Examples \ref{exa:model1} and \ref{ex:model3} scales with $\epsilon^2$. 

\def \scc {0.6}
\def \scl {.8}
\begin{figure}
\unitlength=1cm
\begin{center}
\vspace{3mm}
\psfrag{a}[][][\scl]{ (a)} \psfrag{b}[][][\scl]{ (b)} 
\psfrag{x}[][][\scl]{$x$} \psfrag{u2}[][][\scl]{$u_2$} \psfrag{Wnorm}[][][\scl]{$\| \Delta u \|_{W_2(Q_T)}$} \psfrag{delta}[][][\scl]{$\epsilon$}
\psfrag{dat1}[][][\scl]{$u_2$}
\psfrag{dat2}[][][\scl]{$\tilde u_2$}
\psfrag{dat3}[][][\scl]{$\hat u_2$}
\psfrag{datada1}[][][\scc]{$\|\tilde u - u\|$}
\psfrag{datada2}[][][\scc]{$\|\hat u - \tilde u\|$}
\psfrag{data3}[][][\scc]{$O(\epsilon)$}
\psfrag{data4}[][][\scc]{$O(\epsilon^2)$}
	\includegraphics[width = .4\textwidth]{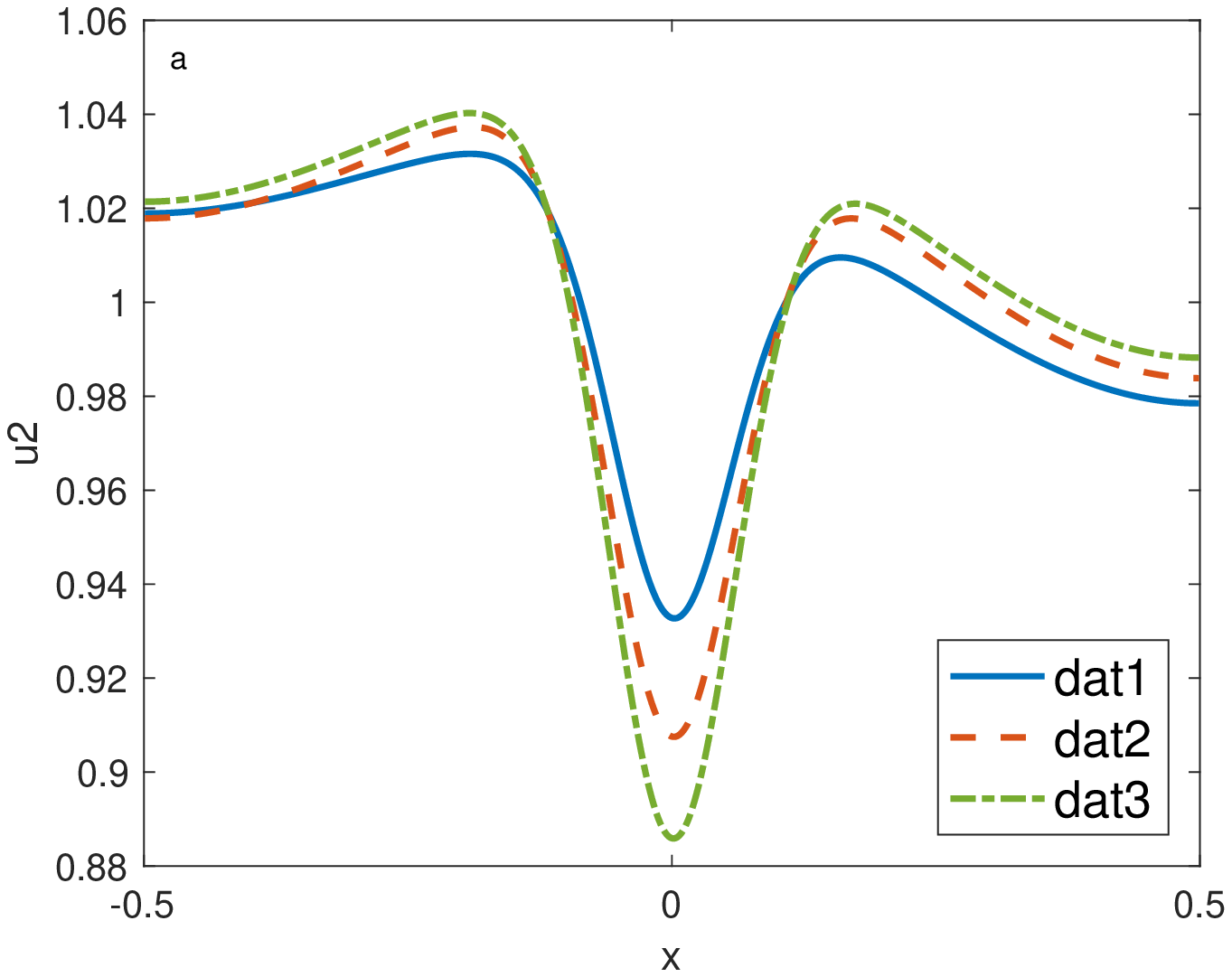} \quad 
	\includegraphics[width = .4\textwidth]{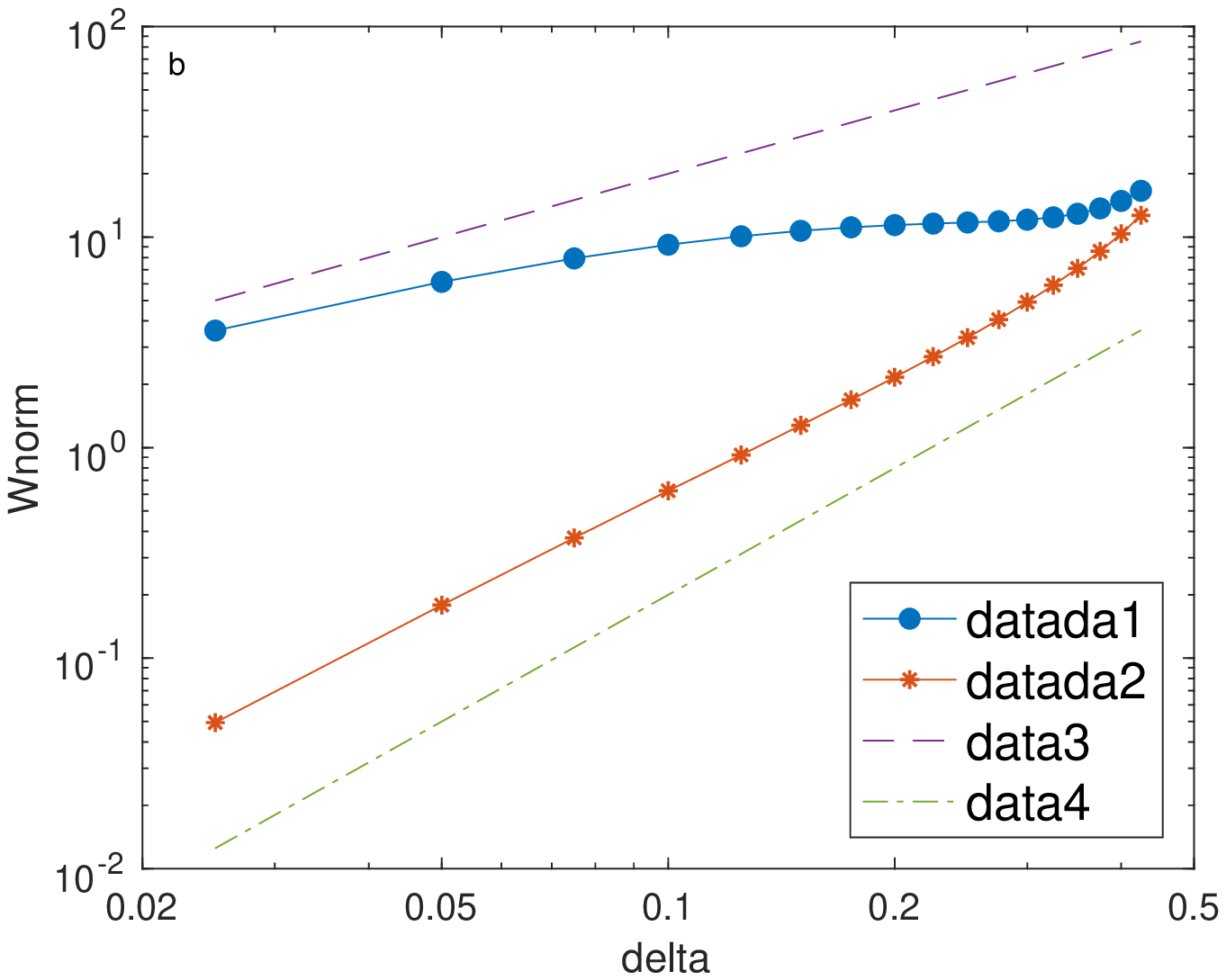}  
\caption{Comparison between the models in Examples \ref{exa:model2}, \ref{exa:model1}, and \ref{ex:model3} for increasing values of $\epsilon$. (a) Second component $u_2$ at time  $t = 0.1$ for $\epsilon = 0.25$ from model \eqref{eq:model2}  ($u_2$), model \eqref{model1} ($\tilde u_2$), and model \eqref{grad_flow_general} ($\hat u_2$).
(b) Norm in $W_2(Q_T)$ of the difference between solutions of models in Examples \ref{exa:model2}  and \ref{exa:model1}, $\|\tilde u-u\|$, and between models in Examples \ref{ex:model3} and \ref{exa:model2}, $\|\hat u-\tilde u\|$. Norm computed using \eqref{discrete_norm} as a function of $\epsilon$. Dash and dot-dash lines show curves $O(\epsilon)$ and $O (\epsilon^2)$ 
for reference. Parameters used: $\bar N_i = 1/2$, $\bar \varepsilon_i = 1$, $D_1=1.5$, $D_2=1$, $V_1(x) = 1-\exp(-120x^2)$ and $V_2(x) = 0$, $d=2$, final time $T = 1$, $J = 500$ and $M = 100$. Initial data as in Figure \ref{fig:fig3}(a).}
\label{fig:fig5}
  \end{center}
\end{figure}

\appendix
\section{Proof of Lemma~\ref{lem:313}} \label{sec:appendix}

Our approach is classical and the parabolic estimate mostly follows
from an elliptic regularity estimate. Yet, for general cross-diffusion
systems, it is well known that such elliptic results do not always
hold, including for quasilinear systems with analytic dependence on
$u$ (see for example \cite{giaquinta_martinazzi} and \cite{stara1995some}). Therefore this result needs to
be proved in the case at hand. 
Some of the more technical
arguments are detailed in well known references (for example \cite{grisvard-1985,troianiello-1987}
concerning elliptic regularity and 
 \cite{dautray-lions-93,evans:98,ladyzhenskaia1988linear,lions-magenes-68} for the parabolic case), so we safely skip a certain number
of intermediate steps, and we give the relevant references.

The following lemma provides the key regularity result. 
\begin{lmm}\label{lem:312}
Given $\omega\in C^{1}\left(\overline{\Omega};\mathbb{R}^+\right)$,
for any $u^{0}\in H^{2}\left(\Omega;\mathbb{R}^{m}\right)$ and 
\[
g_{i}\in C\left(\left[0,T\right];H^{1}\left(\Omega;\mathbb{R}^{d}\right)\right)\cap H^{1}\left(0,T;L^{2}\left(\Omega;\mathbb{R}^{d}\right)\right),\quad i=1,\ldots,m
\]
the weak solution $u$ of 
\begin{align}
\begin{aligned}\omega\partial_{t}u_{i}-\partial_{\alpha}\left[D_{i}^{\alpha\beta}(x,t)\partial_{\beta}u_{i}
+F_{ij}^{\alpha}(x,t)u_{j}+g_{i}^{\alpha}\right] & =0 &  & \text{in}\quad\Omega,\\
\left[D_{i}^{\alpha\beta}(x,t)\partial_{\beta}u_{i}
+F_{ij}^{\alpha}\left(x,t\right]u_{j}+g_{i}^{\alpha}\right]\cdot\nu^\alpha & =0 &  & \text{on}\quad\partial\Omega,\\
u_{i}(0) & =u_{i}^{0} &  & \text{in}\quad\Omega,
\end{aligned}
\label{eq:bvpplusdrift}
\end{align}
for $i=1,\ldots,m$, is unique in $L^{2}\left(0,T;H^{1}\left(\Omega\right)\right)\cap C\left(\left[0,T\right];L^{2}\left(\Omega\right)\right)$
with $\partial_{t}u\in L^{2}(0,T;(H^{1}(\Omega))^{\prime})$.
If the compatibility condition 
\begin{equation}
\left[D_{i}^{\alpha\beta}(x,t)\partial_{\beta}u^0_{i}+F_{ij}^{\alpha}u^0_{j}
+g_{i}^{\alpha}\right]\cdot\nu^\alpha=0 \quad \text{on }\partial\Omega,\quad i=1,\ldots,m\label{eq:compatcond}
\end{equation} 
holds, then $u$ satisfies 
\begin{equation}\label{eq:simple-est-1}
\left\Vert u\right\Vert _{W\left(Q_{T}\right)}\leq \frac{1}{2}C_{T}\left(\left\Vert u^{0}\right\Vert _{H^{2}\left(\Omega\right)}+\left\Vert g\right\Vert _{C\left(\left[0,T\right];H^{1}\left(\Omega\right)\right)\cap H^{1}\left(0,T;L^{2}\left(\Omega\right)\right)}\right),
\end{equation}
where the constant $C_{T}$ is given by \eqref{eq:defC6} and depends
on m, $\lambda$, $T$, the $C^{1}$ norms of $\omega$, $D$, and $F$, and the domain $\Omega$ only.

Furthermore, if $F_{ij}^{\alpha}=D_{i}^{\alpha\beta}\partial_{\beta}V_{i}$
with $V_{i}\in C^{1}\left(\overline{\Omega};\mathbb{R}\right)$, and for each $i$, $D_{i}$ and $V_{i}$ do not depend on time, then 
\begin{equation}\label{eq:simple-est-2}
\left\Vert u\right\Vert _{W\left(Q_{T}\right)}\leq \frac{1}{2}C_{\infty}\left(\left\Vert u^{0}\right\Vert _{H^{2}\left(\Omega\right)}+\left\Vert g\right\Vert _{C\left(\left[0,T\right];H^{1}\left(\Omega\right)\right)\cap H^{1}\left(0,T;L^{2}\left(\Omega\right)\right)}\right),
\end{equation}
 where $C_{\infty}$, given by \eqref{eq:C6prime}, depends on $m,\lambda$,
the $C^{1}$norms of $\omega$, $D$, $F$ and the domain $\Omega$
only. In particular, $C_{\infty}$ is independent of $T$. 
\end{lmm}

\begin{proof}
Note that no coupling appears in \eqref{eq:bvpplusdrift}, therefore
the index $i$ can be dropped, as the result relates to equations,
and not systems. For the purpose of this proof, it is convenient to
modify the formulation of the problem to simplify the computations.
We will write $D=A^{2}$, with $A\in C^{1}\left(\overline{Q_{T}};\mathbb{R}^{d\times d}\right)$
symmetric, positive definite and $A$ satisfies
\begin{equation}
\left\Vert A^{-1}(x,t)\right\Vert _{\infty}\leq\lambda^{-1/2}\text{ in }Q_{T}.\label{eq:ellAbd}
\end{equation}
We write $F=AB,$ and $g=Af$, so that the evolution problem under
consideration can be written under the form 
\begin{equation}
	\label{eq:AB}
	\omega\partial_{t}u-\div\left(A^{2}\nabla u+ABu+Af\right)=0 \quad \text{in }\mathcal{D}^{\prime}\left(\Omega\right).
\end{equation}
The \emph{a priori}\, bounds we will use are 
\begin{align}
\left\Vert A\right\Vert _{L^{\infty}\left(Q_{T}\right)}+\left\Vert \nabla A\right\Vert _{L^{\infty}\left(Q_{T}\right)} & \leq M_{A},\label{eq:bdA}
\\
\left\Vert B\right\Vert _{L^{\infty}\left(Q_{T}\right)}+\left\Vert B\right\Vert _{L^{\infty}\left(Q_{T}\right)}+\left\Vert \nabla B\right\Vert _{L^{\infty}\left(Q_{T}\right)} & \leq M_{B},\label{eq:bdB}
\\
\left\Vert \omega^{-1}\right\Vert _{L^{\infty}\left(Q_{T}\right)}+\left\Vert \omega\right\Vert _{L^{\infty}\left(Q_{T}\right)}+\left\Vert \nabla\omega\right\Vert _{L^{\infty}\left(Q_{T}\right)} & \leq M_{\omega},\label{eq:bdOm}
\end{align}
and 
\begin{equation}
\left\Vert \partial_{t}A\right\Vert _{C\left(\overline{Q_{T}}\right)}+\left\Vert A^{-1}\partial_{t}A\right\Vert _{C\left(\overline{Q_{T}}\right)}+\left\Vert \partial_{t}A^{-1}\right\Vert _{C\left(\overline{Q_{T}}\right)}+\left\Vert \partial_{t}\left(A^{-1}B\right)\right\Vert _{C\left(\overline{Q_{T}}\right)}\leq M_{T}.\label{eq:bdT}
\end{equation}

For $a.e.$ $t\in[0,T]$, we define $\mathcal{A}\left(t,u,v\right):H^{1}(\Omega;\mathbb{R})\times H^{1}(\Omega;\mathbb{R})\to\mathbb{R}$
by 
\begin{equation}
\mathcal{A}(t,u,v)=\int_{\Omega}\left(A^{2}\right)^{\alpha\beta}\left(t,x\right)\partial_{\beta}u\partial_{\alpha}v \, \mathrm{d}x+\int_{\Omega}\left(AB\right)^{\alpha}\left(t,x\right)u\partial_{\alpha}v\, \mathrm{d}x.\label{eq:defnabil}
\end{equation}
Using the \emph{a priori}\, bounds \eqref{eq:bdA} and \eqref{eq:bdB}, we
find the upper bound
\[
\mathcal{A}(t,u,v)\leq M_{A}\left(M_{A}+M_{B}\right)\|u\|_{H^{1}(\Omega)}\|v\|_{H^{1}(\Omega)}.
\]
Furthermore, using \eqref{eq:ellAbd} as well, we have the lower bound
\begin{equation}
\mathcal{A}(t,u,v)\geq\lambda\|u\|_{H^{1}(\Omega)}^{2}-M_{A}M_{B}\|u\|_{L^{2}(\Omega)}\|u\|_{H^{1}(\Omega)}\geq\frac{1}{2}\lambda\|u\|_{H^{1}(\Omega)}^{2}-\frac{1}{2\lambda}M_{A}^2M_{B}^2\|u\|_{L^{2}(\Omega)}^{2}.\label{eq:coercivebd-A}
\end{equation}
We may therefore apply the parabolic version of the Lax--Milgram
Theorem of Lions \cite{brezis,lions-magenes-68} to deduce that there
exists a unique solution of (\ref{eq:bvpplusdrift}) $u\in L^{2}(0,T;H^{1}(\Omega;\mathbb{R}))\cap C([0,T];L^{2}(\Omega;\mathbb{R}))$
with $\partial_{t}u\in L^{2}(0,T;H^{1}(\Omega;\mathbb{R})^{\prime})$.

We now derive an explicit bound. Integrating \eqref{eq:AB} by parts against
$u$ we find
\begin{equation*}
	\partial_{t}\frac{1}{2}\int_{\Omega}\omega u^{2} \,\mathrm{d}x +\int_{\Omega}A^{2}(x,t)\nabla u\cdot\nabla u \,\mathrm{d}x +\int_{\Omega}uAB\cdot\nabla u \,\mathrm{d}x +\int_{\Omega}Af\cdot\nabla u \,\mathrm{d}x =0.
\end{equation*}
Thus, using (\ref{eq:coercivebd-A}) and Cauchy--Schwarz
\[
\partial_{t}\left(\frac{1}{2}\left\Vert \sqrt{\omega}u\right\Vert _{L^{2}\left(\Omega\right)}^{2}\right)+\frac{1}{2}\|A\nabla u\|_{H^{1}(\Omega)}^{2}\leq\left\Vert f\right\Vert _{L^{2}\left(\Omega\right)}^{2}+\left\Vert \omega^{-\frac{1}{2}}B\right\Vert _{L^{\infty}(\Omega)}^{2}\left\Vert \sqrt{\omega}u\right\Vert _{L^{2}\left(\Omega\right)}^{2},
\]
which leads to two bounds
\begin{equation}
\left\Vert u\right\Vert _{C\left(\left[0,T\right],L^{2}\left(\Omega\right)\right)}\leq M_{\omega}^{\frac{1}{2}}\left(\exp\left(\sqrt{2}M_{\omega}M_{B}T\right)\left\Vert f\right\Vert _{L^{2}\left(Q_{T}\right)}+M_{\omega}^{\frac{1}{2}}\left\Vert u^{0}\right\Vert _{L^{2}\left(\Omega\right)}\right),\label{eq:step089}
\end{equation}
and
\begin{equation}
\sqrt{\lambda/2}\|\nabla u\|_{L^{2}\left(0,T;L^{2}(\Omega)\right)}\leq\sqrt{1/2}\|A\nabla u\|_{L^{2}\left(0,T;L^{2}(\Omega)\right)}
\leq\sqrt{\frac{M_\omega}{2}}\left\Vert u^{0}\right\Vert _{L^{2}\left(\Omega\right)}+\left\Vert f\right\Vert _{L^{2}\left(Q_{T}\right)}.\label{eq:step090}
\end{equation}
Note that $\int_{\Omega}u \,\mathrm{d} x =\int_{\Omega}u^{0} \,\mathrm{d}x$ for all times. 
As a result, 
\begin{align}
\left\Vert u\right\Vert _{L^{2}\left(0,T;H^{1}(\Omega)\right)} 
& \leq\sqrt{T}\left|\frac{1}{\left|\Omega\right|}\int_{\Omega}u^{0} \,\mathrm{d}x \right|+\left\Vert u-\frac{1}{\left|\Omega\right|}\int_{\Omega}u \,\mathrm{d}x\right\Vert _{L^{2}\left(0,T;L^{2}(\Omega)\right)}
+\left\Vert \nabla u\right\Vert _{L^{2}\left(0,T;L^{2}(\Omega)\right)} \nonumber \\
 & \leq\sqrt{T}\left|\frac{1}{\left|\Omega\right|}\int_{\Omega}u^{0} \,\mathrm{d}x \right|+(C_{P}\left(\Omega\right)+1)\lambda^{-1/2}\left\Vert A\nabla u\right\Vert _{L^{2}\left(0,T;L^{2}(\Omega)\right)} \nonumber  \\
 & \leq C_{1}\left(\left\Vert u^{0}\right\Vert _{L^{2}\left(\Omega\right)}+\left\Vert f\right\Vert _{L^{2}\left(Q_{T}\right)}\right),\label{eq:step091}
\end{align}
where $C_{P}(\Omega)$ is the Poincar\'e--Wirtinger constant, and
\begin{equation}
C_{1}=\sqrt{T}\left|\Omega\right|^{-1/2}+(C_{P}\left(\Omega\right)+1)\sqrt{\frac{M_\omega}{2\lambda}}.\label{eq:defC1}
\end{equation}
Let us now focus on higher regularity. We are going to show that 
\[
u\in C\left([0,T];H^{2}(\Omega)\right)\cap H^{1}\left(0,T;H^{1}(\Omega)\right).
\]
We write
\begin{equation}
\Phi=A\nabla u+Bu+f.\label{eq:defPhi}
\end{equation}
Thanks to \eqref{eq:step090} and \eqref{eq:step091}, we have 
\[
\|\Phi\|_{L^{2}\left(0,T;L^{2}(\Omega)\right)}\leq C_{2}\left(\left\Vert u^{0}\right\Vert _{L^{2}\left(\Omega\right)}+\left\Vert f\right\Vert _{L^{2}\left(Q_{T}\right)}\right),
\]
with
\begin{equation}
C_{2}=1+\sqrt{M_\omega}+M_{B}C_{1}.\label{eq:defC2}
\end{equation}
Next, we are going to test \eqref{eq:bvpplusdrift} against $\eta = \partial_{t}u-\omega^{-1}\div\left(A\Phi\right)$.
Notice that we have to ensure that $\eta$ is a valid test function. We just sketch the procedure, namely we consider
$\eta_{\tau,h} = \Delta_{\tau}u-\omega^{-1}\Delta_{h}^{\alpha}\left(A^{\alpha\beta}\Phi^{\beta}\right)$,
where the difference quotient time derivative is given by $\Delta_{\tau}u=\left(u\left(\cdot+\tau\right)-u\left(\cdot\right)\right)\tau^{-1}$ 
and difference quotient space derivatives in direction $i$ is given by $\Delta_{-h}^{\alpha}\psi=\left(\psi\left(\cdot+h\text{e}_{\alpha}\right)-\psi\left(\cdot\right)\right)h^{-1}$.
We have to test \eqref{eq:bvpplusdrift} against $\eta_{\tau,h}$ and subsequently pass to the limit for $\tau,h \to 0$, paying attention
to the direction normal to the boundary near $\partial\Omega$.
This step is somewhat technical but straightforward  and it justifies the following calculations rigorously. To simplify
the exposition, we use directly $\partial_{t}u-\omega^{-1}\div\left(A\Phi\right)$
as the test function in the following steps, and obtain 
\begin{equation}
\int_{\Omega}\omega\left(\partial_{t}u\right)^{2} \,\mathrm{d}x +\int_{\Omega}\left(A\Phi\right)\cdot\nabla\left(-\omega^{-1}\div\left(A\Phi\right)\right) \,\mathrm{d}x -2\int_{\Omega}\partial_{t}u \div(A\Phi) \,\mathrm{d}x=0\label{eq:step101-1}
\end{equation}
As $A\Phi\cdot\nu=0$, we find that 
\begin{equation}
\int_{\Omega}\left(A\Phi\right)\cdot\nabla\left(-\omega^{-1}\div\left(A\Phi\right)\right)\,\mathrm{d}x =\int_{\Omega}\omega^{-1}\left(\div\left(A\Phi\right)\right)^{2} \,\mathrm{d}x .\label{eq:step102}
\end{equation}
Let us now turn to the mixed term. We have 
\begin{align} \label{eq:step1011}
-2\int_{\Omega}\partial_{t}u \div (A\Phi)\,\mathrm{d}x  &=2\int_{\Omega}\partial_{t}\left(\left(A^{-1}A\right)\nabla u\right)\cdot\left(A\Phi\right) \,\mathrm{d}x  =2\int_{\Omega}\left[\partial_{t}(A\nabla u)+A\partial_{t}(A^{-1})A\nabla u\right]\cdot\Phi \,\mathrm{d}x \\
 & =2\int_{\Omega}\left[\partial_{t}(\Phi)+A\partial_{t}(A^{-1})\Phi\right]\cdot\Phi \,\mathrm{d}x  -2\int_{\Omega}\left[\partial_{t}(Bu+f)+A\partial_{t}(A^{-1})(Bu+f)\right]\cdot\Phi \,\mathrm{d}x. \nonumber
\end{align}
Inserting  \eqref{eq:step102} and \eqref{eq:step1011} into \eqref{eq:step101-1}
and using Cauchy--Schwarz, we obtain
\begin{align}
\left\Vert \sqrt{\omega}\partial_{t}u\right\Vert _{L^{2}(\Omega)}^{2}&+\int_{\Omega}\omega^{-1}\div\left(A\Phi\right)^{2} \,\mathrm{d}x +\partial_{t}\left\Vert \Phi\right\Vert _{L^{2}\left(\Omega\right)}^{2}
\nonumber  \\
&\leq  2M_{T}M_{A}\left(\left\Vert \Phi\right\Vert _{L^{2}\left(\Omega\right)}^{2}+\left\Vert f\right\Vert _{L^{2}\left(\Omega\right)}\left\Vert \Phi\right\Vert _{L^{2}\left(\Omega\right)}\right)+2\left\Vert \partial_{t}f\right\Vert _{L^{2}\left(\Omega\right)}\left\Vert \Phi\right\Vert _{L^{2}\left(\Omega\right)}
\label{eq:ineq}\\
&\quad + 2M_{T}M_{A}M_{B}\left\Vert u\right\Vert _{L^{2}\left(\Omega\right)}\left\Vert \Phi\right\Vert _{L^{2}\left(\Omega\right)}+2M_{B}M_{\omega}^{\frac{1}{2}}\left\Vert \sqrt{\omega}\partial_{t}u\right\Vert _{L^{2}\left(\Omega\right)}\left\Vert \Phi\right\Vert _{L^{2}\left(\Omega\right)}.
\nonumber 
\end{align}
Using Young's inequality, we recombine inequality \eqref{eq:ineq} to find 
\[
  \frac{1}{2}\left\Vert \sqrt{\omega}\partial_{t}u\right\Vert _{L^{2}\left(\Omega\right)}^{2}+\left\Vert \sqrt{\omega}\div\left(A\Phi\right)\right\Vert _{L^{2}\left(\Omega\right)}^{2}+\partial_{t}\left\Vert \Phi\right\Vert _{L^{2}\left(\Omega\right)}^{2}
\leq  \left(2C_{3}+1\right)\left\Vert \Phi\right\Vert _{L^{2}\left(\Omega\right)}^{2}+M_{B}^{2}\left\Vert u\right\Vert _{L^{2}\left(\Omega\right)}^{2}+\left\Vert \partial_{t}f\right\Vert _{L^{2}\left(\Omega\right)}^{2}+\left\Vert f\right\Vert _{L^{2}\left(\Omega\right)}
\]
with 
\[
C_{3}=2M_{T}M_{A}\left(1+2M_{T}M_{A}\right)+2M_{B}^{2}M_{\omega}.
\]
Integrating in time, we find 
\begin{align*}
\left\Vert \Phi\right\Vert _{C\left(\left[0,T\right],L^{2}\left(\Omega\right)\right)}^{2} & +\frac{1}{2}\left\Vert \sqrt{\omega}\partial_{t}u\right\Vert _{L^{2}\left(Q_{T}\right)}^{2}+\left\Vert \sqrt{\omega}\div\left(A\Phi\right)\right\Vert _{L^{2}\left(Q_{T}\right)}^{2}\\
 & \leq C_{4}\left(\left\Vert u^{0}\right\Vert _{L^{2}\left(\Omega\right)}+\left\Vert f\right\Vert _{L^{2}\left(Q_{T}\right)}\right)^{2}\\
 & \quad +\left\Vert A\nabla u^{0}+Bu^{0}+f\left(t=0\right)\right\Vert _{L^{2}(\Omega)}^{2}+\left\Vert \partial_{t}f\right\Vert _{L^{2}\left(Q_{T}\right)}^{2}+\left\Vert f\right\Vert _{L^{2}\left(Q_{T}\right)}^{2},
\end{align*}
with 
\begin{equation}
C_{4}=\left(2C_{3}+1\right)C_{2}^{2}+M_{B}^{2}C_{1}^{2}.\label{eq:defC4a}
\end{equation}
Let us now check that this allows us to define $\left.\partial_{t}u\right|_{t=0}$
in an appropriate sense. Since
\[
\left\Vert \nabla u\right\Vert _{C\left(\left[0,T\right],L^{2}\left(\Omega\right)\right)}\leq\lambda^{-\frac{1}{2}}\left(\left\Vert \Phi\right\Vert _{C\left(\left[0,T\right],L^{2}\left(\Omega\right)\right)}+M_{B}\left\Vert u\right\Vert _{C\left(\left[0,T\right],L^{2}\left(\Omega\right)\right)}+\left\Vert f\right\Vert _{C\left(\left[0,T\right],L^{2}\left(\Omega\right)\right)}\right),
\]
for any $v\in H^{1}\left(\Omega\right)$, the map 
\[
t\to\int_{\Omega} \left[ A(x,t)\nabla u\cdot\nabla v+Bu \cdot \nabla v+fu\cdot\nabla v \right] \, \text{d}x
\]
is continuous on $[0,T]$. In other words, we define $\left.\partial_{t}u\right|_{t=0}\in\left(H^{1}(\Omega)\right)^{\prime}$
as follows
\begin{align*}
\int_{\Omega}\left.\partial_{t}u\right|_{t=0}v\, \text{d}x & =\lim_{t\downarrow0}\int_{\Omega} \left[ A(x,t)\nabla u\cdot\nabla v+B(x,t)u.\nabla v+f(x,0)\cdot\nabla v \right] \text{d}x\\
 & =\int_{\Omega} \left[A(x,0)\nabla u^{0}\cdot\nabla v+B(x,0)u^{0}.\nabla v+f(x,0)\cdot\nabla v \right] \text{d}x,
\end{align*}
provided that the compatibility condition \eqref{eq:compatcond} holds,
that is, 
\[
\left[A(x,0)\nabla u^{0}-B(x,0)u^{0}-f(x,0)\right]\cdot\nu=0.
\]
An integration by parts then shows that 
\[
\int_{\Omega}\left.\partial_{t}u\right|_{t=0}v \, \text{d}x=\int_{\Omega}\div\left[A(x,0)\nabla u^{0}+B(x,0)u^{0}+f(x,0)\right]v \, \text{d}x,
\]
which, in turn, shows that $\left.\partial_{t}u\right|_{t=0}\in L^{2}(\Omega)$
and
\begin{equation}\label{eq:dtu0}
\left\Vert \left.\partial_{t}u\right|_{t=0}\right\Vert _{L^{2}\left(\Omega\right)}\leq\left(M_{A}+M_{B}\right)\left(\left\Vert u^{0}\right\Vert _{H^{2}\left(\Omega\right)}+\left\Vert f\right\Vert _{C\left(\left[0,T\right];H^{1}\left(\Omega\right)\right)}\right). 
\end{equation}
We now notice that $\partial_{t}u$ is a weak solution of \eqref{eq:bvpplusdrift},
where $f$ is replaced by $\partial_{t}f+\partial_{t}A\nabla u+\partial_{t}Bu$
and $u^{0}$ is replaced by $\left.\partial_{t}u\right|_{t=0}$. 
From \eqref{eq:step091} we obtain
\begin{equation}\label{eq:NfT}
\left\Vert \partial_{t}f+\partial_{t}A\nabla u+\partial_{t}Bu \right\Vert_{L^{2}\left(Q_{T}\right)}
\leq 
\max(M_{T}C_1,1)\left(\left\Vert u^{0}\right\Vert _{L^{2}\left(\Omega\right)}+\left\Vert f\right\Vert _{H^{1}\left(0,T;L^2(\Omega)\right)}\right) 
\end{equation}
Thus \eqref{eq:step090} becomes 
\begin{align*}
\sqrt{\frac{\lambda}{2}}\left\Vert \partial_{t}\nabla u\right\Vert _{L^{2}\left(\left(0,T\right);L^{2}\left(\Omega\right)\right)} 
& \leq
\sqrt{\frac{M_\omega}{2}}
\left\Vert \left.\partial_{t}u\right|_{t=0}\right\Vert _{L^{2}\left(\Omega\right)}+\left\Vert \partial_{t}f+\partial_{t}A\nabla u+\partial_{t}Bu\right\Vert _{L^{2}\left(Q_{T}\right)}\\
 & \leq\sqrt{\frac{M_\omega}{2}}(M_{A}+M_{B})\left(\left\Vert u^{0}\right\Vert _{H^{2}\left(\Omega\right)}+\left\Vert f\right\Vert _{C\left(\left[0,T\right];H^{1}\left(\Omega\right)\right)}\right)\\
& \quad +\max(M_{T}C_1,1)\left(\left\Vert u^{0}\right\Vert _{L^{2}\left(\Omega\right)}+\left\Vert f\right\Vert _{H^{1}\left(0,T;L^2(\Omega)\right)}\right) ,
 \end{align*}
and \eqref{eq:step089} gives 
\begin{align*}
\left\Vert \partial_{t}u\right\Vert _{C\left([0,T];L^{2}\left(\Omega\right)\right)} & \leq M_{\omega}^{\frac{1}{2}}\left[\exp\left(\sqrt{2}M_{\omega}M_{B}T\right)\left\Vert \partial_{t}f+\partial_{t}A\nabla u+\partial_{t}Bu\right\Vert _{L^{2}\left(Q_{T}\right)}+M_{\omega}^{\frac{1}{2}}\left\Vert \left.\partial_{t}u\right|_{t=0}\right\Vert _{L^{2}\left(\Omega\right)}\right]\\
 & \leq C_{5}\left(\left\Vert u^{0}\right\Vert _{H^{2}\left(\Omega\right)}+\left\Vert f\right\Vert _{C\left(\left[0,T\right];H^{1}\left(\Omega\right)\right)}+\left\Vert f\right\Vert _{H^{1}\left(0,T;L^{2}\left(\Omega\right)\right)}\right),
\end{align*}
with 
\begin{equation}
C_{5}=M_{\omega}\left(M_{A}+M_{B}\right) +M_{\omega}^\frac{1}{2} \max(M_{T}C_1,1) \exp\left(\sqrt{2}M_{\omega}M_{B}T\right).\label{eq:defC5-1}
\end{equation}
Finally, we observe that the right-hand side of the identity 
\[
\div\left(A\nabla u\right)=\partial_{t}u-\div\left(Bu+f\right),
\]
belongs to $C\left([0,T];L^{2}\left(\Omega\right)\right)$,
and therefore the left-hand side belongs to the same space. This in turn shows that $u\in H^2(\Omega)$  for any  $t$, in fact $u\in C\left(\left[0,T\right];H^{2}\left(\Omega\right)\right)$,
see, for example, \cite{PLUM199236}, with 
\[
\left\Vert u\right\Vert _{C\left(\left[0,T\right];H^{2}\left(\Omega\right)\right)}\leq 
C(\Omega,M_A,\lambda)\left(C_{5}+M_B C_1\right)\left(\left\Vert u^{0}\right\Vert _{H^{2}\left(\Omega\right)}+\left\Vert f\right\Vert _{C\left(\left[0,T\right];H^{1}\left(\Omega\right)\right)}+\left\Vert f\right\Vert _{H^{1}\left(0,T;L^{2}\left(\Omega\right)\right)}\right).
\]
Altogether, we have shown 
\begin{equation}\label{eq:West-1}
\left\Vert u\right\Vert _{C\left(\left[0,T\right];H^{2}\left(\Omega\right)\right)}+\left\Vert u\right\Vert _{H^{1}\left(0,T;H^{1}\left(\Omega\right)\right)}
\leq 
\frac{1}{2} C_{T}\left(\left\Vert u^{0}\right\Vert _{H^{2}\left(\Omega\right)}+\left\Vert f\right\Vert _{C\left(\left[0,T\right];H^{1}\left(\Omega\right)\right)}+\left\Vert f\right\Vert _{H^{1}\left(0,T;L^{2}\left(\Omega\right)\right)}\right),
\end{equation}
where 
\begin{equation}
C_{T}=2 \left(C(\Omega,M_A,\lambda)\left(C_{5}+M_B C_1\right) +\sqrt{\frac{M_\omega}{\lambda}}\left(M_A+M_B\right) + \sqrt{\frac{2}{\lambda}}\max(M_TC_1,1) \right),\label{eq:defC6}
\end{equation}
and $C_4$ and $C_{5}$ are given by \eqref{eq:defC4a} and \eqref{eq:defC5-1}, respectively, as announced.

Let us now turn to the particular case when $B=A\nabla V$, with $V\in C^{2}\left(\overline{\Omega}\right)$, and $A$ and $V$ are independent of time.
We perform the change of unknown $w=u\exp V$ and, thanks to Lemma~\ref{lem:changeunknown},
we can study the problem satisfied by $w$. 
We have 
\[
\begin{aligned}
\exp\left(-V\right)\partial_{t}w-\div\left[A\exp\left(-V\right)\nabla w+f\right] & =0 &  & \text{in}\quad\Omega,\\
\left[A\exp(-V)\nabla w+f\right]\cdot\nu & =0 &  & \text{on}\quad\partial\Omega,\\
w(0) & =u^{0}\exp(V) &  & \text{in}\quad\Omega,
\end{aligned}
\]
that is, the same system as \eqref{eq:bvpplusdrift} above, with $\omega=\exp(-V)$, $M_{B}=0$
and $M_{T}=0.$ In this case, 
\[
 C_{2}= 1+\sqrt{M_\omega},\quad
 C_{3}= 0, \quad
 C_{4}= C_{2}^{2}, \quad 
 C_{5}= M_{\omega}M_{A}+M_{\omega}^\frac{1}{2},
 \]
and  the constant $C_{T}$ in \eqref{eq:defC6} becomes 
\[
 \tilde{C}'=2 \left(C(\Omega,M_A,\lambda) C_{5} +\sqrt{\frac{M_\omega}{\lambda}}M_A + \sqrt{\frac{2}{\lambda}} \right),
\]
and it does not depend on $T$. Thanks to Lemma \ref{lem:changeunknown},
we find that in terms of $u$ the bound \eqref{eq:West-1} holds with the following constant
\begin{equation}
C_{\infty}=\tilde{C}\left[\left(1+M_{V}^{\prime}\right)^{2}+M_{V}^{\prime\prime}\right]\exp M_{V},\label{eq:C6prime}
\end{equation}
which again is independent of $T$. 
\end{proof}

\begin{rmrk}[Ellipticity bound for $\epsilon$]\label{rem:ellipticitybound}
Suppose that an {\emph{a priori}}  bound for $u$ on $Q_T$ is known, say $u^* = \sup_{Q_T} |u|$. 
For any $\xi_{i}^{\alpha}\in\mathbb{R}^{d\times m},\zeta_j\in\mathbb{R}^{m}$, we have the lower bound
\begin{align}\label{eq:ellipticitybound}
\begin{aligned}
\mathfrak D_{ij}^{\alpha\beta}(t,x,y)\xi_{i}^{\alpha} \xi_{j}^{\beta} 
=  D_i^{\alpha \beta}(t,x)\xi_{i}^{\alpha}\xi_{i}^{\beta} 
+
\epsilon a_{ij}^{\alpha\beta}(t,x)\phi_{ij}^{\alpha\beta}(y)\xi_{i}^{\alpha} \xi_{j}^{\beta}  \geq   (\lambda-\epsilon L_0(u^*) \| a \|_{\infty})|\xi|^{2},
\end{aligned}
\end{align}
where $\| a \|_\infty = \max_{i,j,\alpha,\beta,x} |a_{ij}^{\alpha\beta}(x)|$  and $L_0$ is given in \eqref{eq:DefLi}. 
Therefore, choosing 
\begin{equation}\label{eq:practicalellipticity}
\epsilon < \min\left(\frac{\lambda}{1+ \| a \|_\infty L_0(u^*)},1\right)
\end{equation}
guarantees coercivity, and this is sufficient to ensure existence and uniqueness of weak solutions of \eqref{eq:parab-linear}, and 
consequently of \eqref{eq:1091} via Lax--Milgram lemma. We use relation \eqref{eq:practicalellipticity} to derive an {\emph{a priori}} upper bound for $\epsilon$ in a specific case, see Lemma \ref{lem:second_bound}.
\end{rmrk}

\begin{lmm}
\label{lem:changeunknown} Given $V\in C^{2}\left(\overline{\Omega}\right)$,
the map $u\to u\exp(V)$ is a bi-continuous isomorphism in $C\left(\left[0,T\right];H^{2}\left(\Omega\right)\right)\cap H^{1}\left(0,T;H^{1}\left(\Omega\right)\right).$
The following inequalities hold 
\begin{align*}
\left\Vert u\exp(V)\right\Vert _{W\left(0,T,\Omega\right)} & \leq\left[\left(1+M_{V}^{\prime}\right)^{2}+M_{V}^{\prime\prime}\right]\exp M_{V}\left\Vert u\right\Vert _{W\left(0,T,\Omega\right)},\\
\left\Vert u\right\Vert _{W\left(0,T,\Omega\right)} & \leq\left[\left(1+M_{V}^{\prime}\right)^{2}+M_{V}^{\prime\prime}\right]\exp M_{V}\left\Vert u\exp(V)\right\Vert _{W\left(0,T,\Omega\right)},
\end{align*}
where $M_{V}=\sup_{\Omega}\left|V\right|$, $M_{V}^{\prime}=\sup_{\Omega}\left|\nabla V\right|$
and $M_{V}^{\prime\prime}=\sup_{\Omega}\left|\nabla^{2}V\right|.$
\end{lmm}

\begin{proof}
Note that it is sufficient to prove one inequality, as replacing $V$
by $-V$ changes the map to its inverse.
Indeed, we have 
\begin{align*}
\left\Vert u \exp(V)\right\Vert _{L^{2}(\Omega)} & \leq\exp M_{V}\left\Vert u\right\Vert _{L^{2}(\Omega)},\\
\left\Vert u \exp(V)\right\Vert _{H^{1}(\Omega)} & \leq\left(1+M_{V}^{\prime}\right)\exp M_{V}\left\Vert u\right\Vert _{H^{1}(\Omega)},\\
\left\Vert u \exp(V)\right\Vert _{H^{2}(\Omega)} & \leq\left[\left(1+M_{V}^{\prime}\right)^{2}+M_{V}^{\prime\prime}\right]\exp M_{V}\left\Vert u\right\Vert _{H^{2}(\Omega)}.
\end{align*}
\end{proof}

The second step in the proof of Lemma \ref{lem:313} concerns the regularity of the forcing term $f$, which coincides with the regularity of the cross-diffusion term, provided that $\yply$ and $u$ are in $W\left(Q_T\right)$. 
\begin{lmm}\label{lem:FNL}
The map
\begin{align}
\label{eq:defFNL}
\begin{aligned}
	P:Q_T \times C^{\infty}\left(Q_{T};\mathbb{R}^{m}\right)^2 \quad  &\to \quad  C^{2}\left(Q_{T}:\mathbb{R}^{m \times d}\right) \\
\left(t,x,\yply,u\right)\quad  & \to \quad    a_{ij}^{\alpha\beta}(t,x)\phi_{ij}^{\alpha\beta}\left(\yply\right)\partial_{\beta}u_{j}+b_{ij}^{\alpha}(t,x)\psi_{ij}^{\alpha}\left(\yply\right)u_{j}, 
\end{aligned}
\end{align}
has the following property
\[
P\left(Q_T \times W\left(Q_{T}\right)\times W\left(Q_{T}\right)\right)\subset C\left(\left[0,T\right];H^{1}\left(\Omega;\mathbb{R}^{m}\right)\right)\cap H^{1}\left(0,T;L^{2}\left(\Omega;\mathbb{R}^{m}\right)\right).
\]
 Furthermore, there holds 
\[
\sup_{\left[0,T\right]}\left(\left\Vert \nabla P_{i}(t,x,\yply,u)\right\Vert _{L^{2}(\Omega)}
+
\left\Vert P_{i}\left(t,x,\yply,u\right)\right\Vert _{L^{2}\left(\Omega\right)}
\right)+\left\Vert \partial_{t}P_{i}\left(t,x,\yply,u\right)\right\Vert _{L^{2}\left(Q_{T}\right)} 
\leq K_0 
\left(\left\Vert \yply\right\Vert _{W\left(Q_{T}\right)} \right) \left\Vert u\right\Vert _{W\left(Q_{T}\right)},
\]
where $K_0$ is given by \eqref{eq:defKo}. 
\end{lmm}

\begin{proof}
Note that $L^{\infty}\left(Q_{T}\right)\subset C\left(\left[0,T\right];H^{2}\left(\Omega;\mathbb{R}^{m}\right)\right)$. Therefore 
\[
\sup_{Q_{T}}\left|\yply\right|\leq {C_{S}^{\infty}}\left\Vert \yply\right\Vert _{W\left(Q_{T}\right)},
\]
where 
\begin{equation}\label{eq:defCSinfty}
C_{S}^{\infty}=C(H^2(\Omega)\hookrightarrow L^\infty(\Omega)) 
\end{equation}
 is the Sobolev constant associated to the embedding of $H^2(\Omega)$ into $L^\infty(\Omega)$, and depends on $\Omega$ and $d$.  We compute the following bounds for $P$
\begin{align*}
|P_{i}(t,x,\yply,u)| & \leq\sup_{\Omega\times[0,\infty)}(|a|,|b|)L_{0}\left(\sup_{Q_{T}}\left|\yply\right|\right)\left(\left|\nabla u\right|+\left|u\right|\right),
\\
\left\Vert P_{i}\left(t,x,\yply,u\right)\right\Vert _{L^{2}\left(\Omega\right)} & \leq M \left\Vert u\right\Vert _{H^{1}\left(\Omega\right)}
 \leq ML_{0}\left({C_{S}^{\infty}}\left\Vert \yply\right\Vert _{W\left(Q_{T}\right)}\right) \left\Vert u\right\Vert _{W\left(Q_{T}\right)},
\end{align*}
for all $t\in\left[0,T\right]$.
Similarly, for the spatial derivatives of $P$ we have
\begin{align*}
\left|\partial_{\alpha}P_{i}(t,x,\yply,u)\right| &\leq  M\left(L_{0}\left({C_{S}^{\infty}}\left\Vert \yply\right\Vert _{W\left(Q_{T}\right)}\right)+L_{1}\left({C_{S}^{\infty}}\left\Vert \yply\right\Vert _{W\left(Q_{T}\right)}\right)\left|\nabla\yply\right|\right)\left(\left|\nabla u\right|+\left|u\right|\right)\\
 & \quad +ML_{0}\left({C_{S}^{\infty}}\left\Vert \yply\right\Vert _{W\left(Q_{T}\right)}\right)\left(\left|\nabla^{2}u\right|+\left|\nabla u\right|\right).
\end{align*}
Therefore, using Cauchy--Schwarz and the Sobolev embedding $H^{1}\left(\Omega\right)\hookrightarrow L^{4}\left(\Omega\right)$  we find
\begin{align*}
 \left\Vert \nabla P_{i}(t,x,\yply,u )\right\Vert _{L^{2}(\Omega)} &\leq 2 M L_{0}\left({C_{S}^{\infty}}\left\Vert \yply\right\Vert _{W\left(Q_{T}\right)}\right)\left\Vert u\right\Vert _{W\left(Q_{T}\right)}  +ML_{1}\left({C_{S}^{\infty}}\left\Vert \yply\right\Vert _{W\left(Q_{T}\right)}\right)\left\Vert \nabla\yply\right\Vert _{L^{4}}\left(\left\Vert \nabla u\right\Vert _{L^{4}}+\left\Vert u\right\Vert _{L^{4}}\right)\\
 & \leq M\left[2L_{0}\left({C_{S}^{\infty}}\left\Vert \yply\right\Vert _{W\left(Q_{T}\right)}\right)+C_S^2 L_{1}\left({C_{S}^{\infty}}\left\Vert \yply\right\Vert _{W\left(Q_{T}\right)}\right)\left\Vert \yply\right\Vert _{W\left(Q_{T}\right)}\right]\left\Vert u\right\Vert _{W\left(Q_{T}\right)},
\end{align*}
where $C_S^2$ is defined by \eqref{eq:defCS2}. This shows that $P_{i}\left(t,x,\yply,u\right)\in C\left(\left[0,T\right];H^{1}\left(\Omega\right)\right).$
Finally, for the time derivative we obtain
\begin{align*}
 \left|\partial_{t}P_{i}(t,x,\yply,u)\right| &\leq M\left[ L_{0}\left({C_{S}^{\infty}}\left\Vert \yply\right\Vert _{W\left(Q_{T}\right)}\right)
 +L_{1}\left({C_{S}^{\infty}}\left\Vert \yply\right\Vert _{W\left(Q_{T}\right)}\right)\left|\partial_{t}\yply\right|\right]\left(|\nabla u|+|u|\right)\\
 & \quad +ML_{0}\left({C_{S}^{\infty}} \left\Vert \yply\right\Vert _{W(Q_{T})}\right)\left(\left|\nabla\partial_{t}u\right|+\left|\partial_{t}u\right|\right),
\end{align*}
and
\[
\left\Vert \partial_{t}P_{i}(t,x,\yply,u)\right\Vert _{L^{2}\left(Q_{T}\right)}\leq M\left[2L_{0}\left({C_{S}^{\infty}}\left\Vert \yply\right\Vert _{W\left(Q_{T}\right)}\right)
+L_{1}\left({C_{S}^{\infty}}\left\Vert \yply\right\Vert _{W\left(Q_{T}\right)}\right)\left\Vert \partial_{t}\yply\right\Vert_{L^2(Q_T)}\right]\left\Vert u\right\Vert _{W\left(Q_{T}\right)}.
\]
Altogether we have shown that
$$
 \sup_{\left[0,T\right]}\left(\left\Vert \nabla P_{i}(t,x,\yply,u)\right\Vert _{L^{2}(\Omega)} 
+
\left\Vert P_{i}(t,x,\yply,u)\right\Vert _{L^{2}\left(\Omega\right)}
\right)+\left\Vert \partial_{t}P_{i}(t,x,\yply,u)\right\Vert _{L^{2}\left(Q_{T}\right)}
\leq K_0 \left( \left\Vert \yply\right\Vert _{W\left(Q_{T}\right)}\right)   \left\Vert u\right\Vert _{W\left(Q_{T}\right)},
$$
where $K_0$ is defined by \eqref{eq:defKo}, as announced. 
\end{proof}

\begin{proof}[Proof of Lemma \ref{lem:313}]
We write 
\[
\partial_{\alpha}\left[\mathfrak D_{ij}^{\alpha\beta}(t,x,\yply)\partial_{\beta}u_{j}-\mathfrak F_{ij}^{\alpha}(t,x,\yply)u_{j}+f_{i}^{\alpha}\right]=\partial_{\alpha}\left[D_{i}^{\alpha\beta}(t,x)\partial_{\beta}u_{j}-F_{i}^{\alpha}(t,x)u_{j}+g_{i}^{\alpha}\right],
\]
with $g_{i}^{\alpha}=f_{i}^{\alpha}+\epsilon P_{i}^{\alpha}(t,x,\yply,u)$, and $P$ given by \eqref{eq:defFNL}. 
Lemma \ref{lem:312} shows that 
\[
\left\Vert u\right\Vert _{W\left(Q_{T}\right)}\leq \frac{1}{2}  C_{T}\left(\left\Vert u^{0}\right\Vert _{H^{2}\left(\Omega\right)}+\left\Vert g\right\Vert _{C\left(\left[0,T\right];H^{1}\left(\Omega\right)\right)\cap H^{1}\left(0,T;L^{2}\left(\Omega\right)\right)}\right),
\]
and 
\begin{equation*}
\Vert g\Vert _{C([0,T];H^{1}(\Omega))\cap H^{1}(0,T;L^{2}(\Omega))} 
 \leq\left\Vert f\right\Vert _{C\left(\left[0,T\right];H^{1}\left(\Omega\right)\right)\cap H^{1}\left(0,T;L^{2}\left(\Omega\right)\right)} +\epsilon\left\Vert P(t,x,\yply,u)\right\Vert _{C\left(\left[0,T\right];H^{1}\left(\Omega\right)\right)\cap H^{1}\left(0,T;L^{2}\left(\Omega\right)\right)}.
\end{equation*}

Thanks to Lemma \ref{lem:FNL}, there holds 
\[
  \sup_{\left[0,T\right]}\left(\left\Vert \nabla P_{i}\left(t, x,\yply,u\right)\right\Vert _{L^{2}\left(\Omega\right)}\right)+\left\Vert \partial_{t}P_{i}\left(t, x, \yply,u\right)\right\Vert _{L^{2}\left(Q_{T}\right)}
\leq  
K_0\left( \left\Vert \yply\right\Vert _{W\left(Q_{T}\right)}\right)
\left\Vert u\right\Vert _{W\left(Q_{T}\right)},
\]
and therefore 
\[
 \left\Vert u\right\Vert _{W\left(Q_{T}\right)}\left[1-K_0\big( \Vert \yply \Vert _{W(Q_{T})}\big)
\right]
  \leq \frac{1}{2} C_{T}\left(\left\Vert u^{0}\right\Vert _{H^{2}\left(\Omega\right)}
  +\left\Vert f\right\Vert _{C\left(\left[0,T\right];H^{1}\left(\Omega\right)\right)\cap H^{1}\left(0,T;L^{2}\left(\Omega\right)\right)}\right),
\]
 which is our thesis, as thanks to the Fredholm Alternative, boundedness implies existence and uniqueness.  The proof in the time independent case is analogous and $C_T$ is replaced by $C_\infty$.
\end{proof}

\textbf{Acknowledgements.}
The authors are very thankful for the detailed comments and suggestions of the referees which have significantly improved the quality of this manuscript. The third author was visiting Laboratoire Jacques-Louis Lions  during the final stage of this paper, and he is very grateful for the hospitality and warmth of his hosts. 

\def\cprime{$'$}


\end{document}